\documentclass[12pt, reqno]{amsart}
\usepackage{amssymb}
\usepackage{amsmath}
\usepackage{amscd}
\usepackage{mathrsfs}
\usepackage{float}
\usepackage{graphicx}
\usepackage{amsfonts}
\usepackage{color}
\usepackage{pb-diagram}
\usepackage[utf8]{inputenc}
\usepackage[T1]{fontenc}
\usepackage{multicol} 
\newcommand{\Y}{{\rm Y}_{d,n}^{\mathtt B}}
\newcommand{\Z}{{\rm Y}_{d,n-1}^{\mathtt B}}

\newcommand{\s}{\mathtt s}
\newcommand{\R}{\mathtt r}

\newcommand{\U}{\mathsf{u}}
\newcommand{\V}{\mathsf{v}}

\newcommand{\vnor}{v_i^r\otimes v_j^s}
\newcommand{\vinv}{v_j^s\otimes v_i^r}

\newcommand{\vinvn}{v_{-j}^{s}\otimes v_i^r}
\newcommand{\vnorn}{v_{-i}^{r}\otimes v_j^s}

\newtheorem{theorem}{Theorem}
\newtheorem{corollary}{Corollary}
\newtheorem{lemma}{Lemma}
\newtheorem{proposition}{Proposition}
\newtheorem{definition}{Definition}
\newtheorem{remark}{Remark}
\newtheorem{example}{Example}
\newtheorem*{note}{Note}

\begin{document}

\title{A framization of the Hecke algebra of type $\mathtt B$}

\author{Marcelo Flores}
\address{Instituto de Matem\'{a}ticas, Universidad de Valpara\'{i}so,
Gran Breta\~{n}a 1091, Valpara\'{i}so, Chile.}
\email{marcelo.flores@uv.cl}

\author{Jes\'us Juyumaya}
\address{Instituto de Matem\'{a}ticas, Universidad de Valpara\'{i}so,
Gran Breta\~{n}a 1091, Valpara\'{i}so, Chile.}
\email{juyumaya@gmail.com}

\author{Sofia Lambropoulou}
\address{School of Applied Mathematical and Physical Sciences,
National Technical University of Athens,
Zografou campus, GR-157 80 Athens, Greece.}
\email{sofia@math.ntua.gr}
\urladdr{http://www.math.ntua.gr/$\tilde{~}$sofia}

\thanks{The first author was partially supported by Conicyt (Programa de Inserci\'on de Capital Humano Avanzado, PAI 79140019). The second
 author was partially supported by Fondecyt (Grant No. 1141254). The research of the third author has been cofinanced by the European Union (European Social Fund—ESF) and Greek national funds through the Operational Program “Education and Lifelong Learning” of the National Strategic Reference Framework (NSRF)—Research Funding Program: THALES: Reinforcement of the interdisciplinary and/or inter-institutional research and innovation.
}


\date{}

\begin{abstract}
In this article we introduce a framization of the Hecke algebra of type $\mathtt B$.
For this framization we construct a faithful tensorial representation and two linear bases.
We finally construct a Markov trace on these algebras and from this trace we derive isotopy invariants for framed and classical knots and links in the solid torus.
\end{abstract}

\maketitle

\setcounter{tocdepth}{1}
\begin{center}
\begin{minipage}{10cm}
{\scriptsize\tableofcontents}
\end{minipage}
\end{center}

\section*{Introduction}
The idea of framization of a knot algebra (Hecke algebras and BMW--algebra among others) was introduced by the two last authors in \cite{jula7} and consists in adding framing generators to the defining generators of the knot algebra with the aim of finding new invariants of classical links or, more generally,  invariants of knot-like objects. The Yokonuma--Hecke algebra  is the prototype of framization; indeed this algebra, introduced by T. Yokonuma \cite{yo} in the context of representations of Chevalley groups, can be thought of as a framization of the Hecke algebra of type $A$.

\smallbreak

More precisely, the Yokonuma--Hecke algebra supports a Markov trace \cite{juJKTR} and then it becomes  a peculiar knot algebra considering that,  by using the Jones' recipe,   one can construct  invariants for: framed links \cite{jula5}, classical links \cite{jula4} and singular links \cite{jula3}. It is worth mentioning that recently it was proved that the invariants for classical links constructed in \cite{jula4} are not topologically equivalent either to the Homflypt polynomial or to the Kauffman polynomial, see \cite{chjukala}.

\smallbreak

On the other hand, Jones raised that his recipe for the construction of the Homflypt  polynomial might  be used for  Hecke algebra not only of type $\mathtt A$, cf.  \cite[p.336]{jo}.  Then,  the third  author studied the Jones recipe by using the Hecke algebra of type $\mathtt B$; namely, in  \cite{la1} she constructed the analogue of the Homflypt polynomial for oriented knots and links inside the solid torus, see also \cite{gela}. Further, in  \cite{la2} Lambropoulou constructed all possible analogues of the Homflypt polynomial in the solid torus from the Ariki-Koike algebras and the affine Hecke algebras of type $\mathtt A$.

\smallbreak

The purpose of this article is to introduce and to start a systematic study of a framization of the Hecke algebras of type $\mathtt B$, denoted by $\Y(\U,\V)$,  with the principal objective to explore their usefulness  in knot theory. Thus, having in mind both the role of the Hecke algebra
of type $\mathtt B$  \cite{gela} and  the
Yokonuma--Hecke algebra of type $\mathtt A$ \cite{jula5, jula4, jula3} in knot theory,  it is natural to define, by using the Jones' recipe applied to  $\Y(\U,\V)$,  invariants  in the solid torus for: classical  knots, framed knots and singular knots. For these purposes a first key point is to prove that the algebra $\Y(\U, \V)$ supports a Markov trace. In fact, this is one of the main results proved in the present article.

\smallbreak

 In \cite{chpoIMRN}, M. Chlouveraki and L. Poulain D' Andecy have introduced the affine and cyclotomic Yokonuma--Hecke algebras. In the context of framization \cite{jula7}, the definition of the algebras introduced by Chlouveraki and Poulain D' Andecy can be understood by adding framing generators and making the framization according to the formula of framization of the generators of the Yokonuma--Hecke algebra of type $\mathtt{A}$, that is, only of the braiding generators. Now, the Hecke algebra of type $\mathtt{B}$ is a particular case of cyclotomic Hecke algebra.  The framization of the Hecke algebra of type $\mathtt{B}$ proposed here makes the framization of all generators of the Hecke algebra of type $\mathtt{B}$; in particular, also of the special \lq loop\rq\ generator of the algebra.

\smallbreak

Before giving the organization of the article we note that, by taking into account  the various articles generated recently from  the algebra of Yokonuma--Hecke of type $\mathtt A$  (view for example \cite{esry,cui1,cui2,chporep} among others), the framization proposed here  indicates  that the algebra
$\Y(\U, \V)$ should be interesting in itself.

\smallbreak

The  article is organized as follows. In Section 1 we introduce our notation and explain the background notions. In Section 2 we define our framizations for the Coxeter group of type $\mathtt{B}$, for the Artin braid group of type $\mathtt{B}$ and for the Hecke algebra of type $\mathtt{B}$, $\Y$. In Section 3 we construct a  tensorial representation for the algebra $\Y:=\Y(\U,\V)$. In Section 4 we find linear bases for $\Y$, one of which  is used in Section 5 for constructing a Markov trace on the algebras $\Y$. Finally, in Section 6 necessary and sufficient conditions are given for the trace parameters in order to proceed with the construction of topological invariants of framed and classical knots and links in the solid torus (Section 7).

\section{Notation and background}

In this section we review known results, necessary for this paper, and we also fix the following terminology and notations that will be used  along the paper:
\begin{enumerate}

\item[--] The letters $\U$ and $\V$ denote two indeterminates. And we denote by  ${\Bbb K}$, the field of rational functions ${\Bbb C}(\U,\V)$.

\item[--]The term {\it algebra} means unital associative algebra over  ${\Bbb K}$

\item[--] For a finite group $G$, ${\Bbb K}[G]$ denotes the group algebra of $G$

\item[--] The letters $n$ and $d$ denote two fixed  positive integers

\item[--] We denote by  $\omega$ a fixed  primitive $d$--th root of unity

\item[--] We denote by ${\Bbb Z}/d{\Bbb Z}$ the group of integers modulo $d$, $\{0, 1,\ldots, d-1\}$, and by $C_d$ the cyclic group of order $d$, $\langle t \, ;\, t^d=1\rangle$. Note  that   $C_d\cong {\Bbb Z}/d{\Bbb Z}$.

\item[--] As usual, we denote by $\ell$ the length function associated to the Coxeter groups.

\end{enumerate}

\subsection{\it Braid groups of type $\mathtt{A}$}

The finite Coxeter group of type ${\mathtt A}_n$ ($n\geq 2$)  can be realized  as the symmetric group  on the set  $\{1, \ldots , n\}$. Set $s_i$ the elementary transposition $(i, i+1)$, so the Coxeter presentation of $S_n$ is encoded in the following Dynkin diagram:

\begin{center}
\setlength\unitlength{0.2ex}
\begin{picture}(350,40)
\put(82,20){$s_1$}
\put(120,20){$s_{2}$}
\put(200,20){$s_{n-2}$}
\put(240,20){$s_{n-1}$}

\put(85,10){\circle{5}}
\put(87.5,9){\line(1,0){35}}
\put(125,10){\circle{5}}
\put(127.5,10){\line(1,0){10}}

\put(145,10){\circle*{2}}
\put(165,10){\circle*{2}}
\put(185,10){\circle*{2}}

\put(205,10){\circle{5}}
\put(207.5,10){\line(1,0){35}}
\put(245,10){\circle{5}}
\put(192.5,10){\line(1,0){10}}


\end{picture}
\end{center}

Then, the Artin braid group, $B_n$, associated to $S_n$ is generated by the elementary braidings $\sigma_1 , \ldots , \sigma_{n-1}$, which satisfy the following relations:

 \begin{equation}\label{braidr}
\begin{array}{rcll}
 \sigma_i \sigma_j & =  & \sigma_j \sigma_i & \text{ for} \quad \vert i-j\vert >1\\
  \sigma_i \sigma_j \sigma_i & = & \sigma_j \sigma_i \sigma_j & \text{ for} \quad \vert i-j\vert = 1.
  \end{array}
\end{equation}

\bigbreak
The \textit{framed braid group}, ${\mathcal F}_n$, is defined as the group presented by the braiding generators  $\sigma_1, \ldots , \sigma_{n-1}$ and the framing generators $t_1, \ldots , t_n$ subject to the relations (\ref{braidr}) together with the relations:

\begin{equation}\label{fbraidA}
\begin{array}{rclr}
 t_i t_j & =  & t_jt_i & \text{for all}  \quad  i, j\\
 t_j\sigma_i   & = &  \sigma_i t_{s_i(j)}&\text{for all}\quad i, j
  \end{array}
\end{equation}

\bigbreak
Notice that ${\mathcal F}_n$ is isomorphic to the  wreath product ${\mathbb Z}\wr B_n$. The $d$--\textit{modular framed braid group}  ${\mathcal F}_{d, n}$ is defined by adding to the above defining presentation of  ${\mathcal F}_n$ the relations $t_i^d = 1$. Hence, ${\mathcal F}_{d,n}\cong \left({\mathbb Z}/d {\mathbb Z} \right)\wr B_n$.

\smallbreak

  It is convenient to write the elements of ${\mathcal F}_{n}$ in the split form $\sigma t_1^{m_1}\ldots t_n^{m_n}$, where the $m_i$'s are integers (called the framings)   and $\sigma\in B_n$. A framed braid can be represented as a usual geometric braid on $n$ strands by attaching the respective framing to each strand.

For details and the above geometric interpretation of the framed braid group see \cite{kosm}. See also \cite{jula2, jula5}.

\subsection{\it Braid groups of type $\mathtt{B}$}\label{coxeterB}

Set $n\geq 2$. Let us denote by $W_n$ the Coxeter group of type ${\mathtt  B}_n$. This is the finite Coxeter group
 associated to the following Dynkin diagram

Define $\R_k=\s_{k-1}\ldots \s_1 \R_1 \s_1\ldots \s_{k-1}$ for $2\leq k\leq n$.
It is known, see \cite{gela},  that every element $w\in W_n$ can be written uniquely as $w=w_1\ldots w_n$ with $w_k\in \mathtt{N}_k$, $1\leq k\leq n$, where
\begin{equation}\label{NWn}
\mathtt{N}_k:=\left\{
1, \R_{k},
\s_{k-1}\cdots \s_{i},
\s_{k-1}\cdots \s_{i}\R_{i}\, ;\, 1\leq i \leq k-1
\right\}
\end{equation}
Furthermore, this expression for $w$ is reduced. Hence, we have $\ell(w)=\ell(w_1)+\cdots +\ell(w_n)$.
\bigbreak
The corresponding \textit{braid group of type} ${\mathtt  B}_{n}$ associated to $W_n$, is defined as the group $\widetilde{W}_n$ generated  by $\rho_1 , \sigma_1 ,\ldots ,\sigma_{n-1}$ subject to the following relations
 \begin{equation}\label{braidB}
\begin{array}{rcll}
 \sigma_i \sigma_j & =  & \sigma_j \sigma_i & \text{ for} \quad \vert i-j\vert >1\\
  \sigma_i \sigma_j \sigma_i & = & \sigma_j \sigma_i \sigma_j & \text{ for} \quad \vert i-j\vert = 1\\
   \rho_1\sigma_i&=&\sigma_i\rho_1 &\text{ for}\quad i>1\\
\rho_1 \sigma_1 \rho_1\sigma_1 & = & \sigma_1 \rho_1 \sigma_1\rho_1. &
  \end{array}
\end{equation}

Geometrically, braids of type ${\mathtt  B}_{n}$ can be viewed as classical braids of type ${\mathtt  A}_{n+1}$ with $n+1$ strands, such that the first strand is identically fixed. This is called `the fixed strand'. The 2nd, \ldots, $(n+1)$st strands are renamed from 1 to $n$ and they are called `the moving strands'. The `loop' generator $\rho_1$ stands for the looping of the first moving strand around the fixed strand in the right-handed sense. In Figure \ref{bbraid} we illustrate a braid of type ${\mathtt  B}_5$.

\begin{figure}[h]
\begin{center}

  \includegraphics{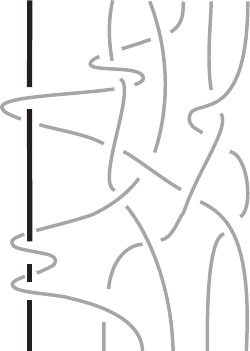}
	\caption{A braid of ${\mathtt  B}_5$-type.}\label{bbraid}
	\end{center}
 \end{figure}

\subsection{}\label{realized1}

The group $W_n$ can be realized as a subgroup of the permutation group of the set  $X_n:=\{-n, \ldots , -2, -1, 1, 2, \ldots, n\}$. More precisely, the elements of $W_n$ are the permutations $w$ such  that $w(-m) = - w(m)$, for all $m \in X_n$. For example
$$
  \s_i \ \ \text{is realized as} \ \ \left(
                       \begin{array}{ccccccccccc}
                         -n &\dots  &-i-1  &-i  & -i+1 &\dots  & i-1 & i & i+1 &\dots &n \\
                          -n &\dots  & -i &-i-1 & -i+1 &\dots  & i-1 & i+1 & i &\dots &n  \\
                       \end{array}
                     \right)
$$
and
$$
  \R_1 \ \ \text{is realized as} \ \ \left(
                       \begin{array}{cccrrccc}
                         -n & \dots & -2 & -1 & 1 & 2 & \dots & n \\
                         -n & \dots & -2 & 1 & -1 & 2 & \dots & n\\
                       \end{array}
\right).
$$

Further, the elements of $W_n$ can be parameterized  by the elements of $X_n^n$ (see \cite[Lemma 1.2.1]{gr}). More precisely,   the element $w\in W_n$ corresponds to the element $(m_1, \ldots ,m_n)\in X_n^n$ such that  $m_i= w(i)$. Then,   we have
that $\s_i$ is parameterized by $(1,2,\dots,i+1,i,\dots, n)$ and $\R_1$ is parameterized by $ (-1,2,\dots, n)$. More generally, if  $w\in W_n$ is parameterized by $(m_1,\dots, m_n)\in X_n^n$, then
\begin{equation}\label{parametrized}
\begin{array}{ccl}
w\R_1 & \quad \text{is parameterized by}&\quad (-m_1,m_2,\dots, m_n)\\
w\s_i &\quad \text{is parameterized by}&\quad (m_1,\dots,m_{i+1},m_{i},\dots,m_n ).
\end{array}
\end{equation}

\begin{lemma}\label{coxeter}\cite[Lemma 1.2.2]{gr}
Let $w\in W_n$ parameterized  by $(m_1,\dots, m_n)\in X_n^n$. Then $\ell(w \s_i)=\ell(w)+1$ if and only if $m_i<m_{i+1}$ and $\ell(w\R_1)=\ell(w)+1$ if and only if $m_1>0$.
 \end{lemma}
\begin{example} \rm
  Set $n=3$ and $w=\s_1\R_1$. Then we have that $\ell(w\R_1)<\ell(w)$ and $\ell(w\s_1)>\ell(w)$, and $w$ is represented by $(-2,1,3)$.
\end{example}

\begin{remark}[Symmetric braids]\label{symbraids1} \rm
The above realization of the $W_n$  can be lifted also to the braid level. Indeed in \cite{tomd} tom Dieck defines the \textit{symmetric braids} (denoted by $ZB_n$) and proves that $\widetilde{W}_n$ is isomorphic to this group of braids. Specifically tom Dieck considers the braids in $\mathbb{R}\times [0,1]$ with strands between $X_n\times \{ 0\} \times \{ 0\}$ and $X_n\times \{ 0\} \times \{ 1\}$ which are symmetric about the axis $(0,0)\times [0,1]$. Moreover, the group of the symmetric braids is generated by the elements $z_0,z_1,\dots, z_{n-1}$ represented graphically as follows.

\begin{figure}[h]
\begin{center}
  \includegraphics{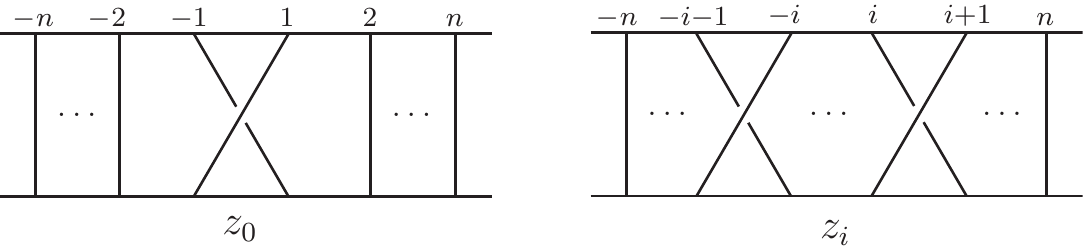}
	\caption{The generators of the symmetric braid group.}\label{symbraids}
	\end{center}
 \end{figure}

An isomorphism between the groups $\widetilde{W}_n$ and $ZB_n$ is induced by the mapping:
$\rho_1\mapsto z_0$, $\sigma_i\mapsto z_i$ for $1\leq i \leq n-1$.
\end{remark}


\subsection{}

The Hecke algebra of type $\mathtt{B}_n$, denoted ${\rm H}_n(\U, \V)$, is the algebra generated by $h_0, h_1, \ldots, h_{n-1}$ subject  to the following relations:
$$
\begin{array}{cccl}
h_ih_j & = &  h_j h_i & \text{for all}\quad \vert i - j\vert >1 \\
h_ih_{i+1}  h_i  & = &  h_{i+1} h_i  h_{i+1}&\text{for all}\quad i =1, \ldots , n-2\\
h_1  h_0 h_1 h_0 & = & h_0h_1 h_0h_1 & \\
h_i^2 & = & 1 + (\U-\U^{-1})h_i &  \text{for all}\quad i\\
h_0^2 & = & 1 + (\V-\V^{-1})h_0.&
\end{array}
$$
It is well known that the dimension of ${\rm H}_n(\U, \V)$ is $2^nn!$ and clearly for $\U = \V = 1$ it  coincides with $\mathbb{K}[W_n]$.

\section{Framizations of type ${\mathtt B}$} \label{framizations}

In this section  we introduce the main object studied in the paper, that is, a framization of the Hecke algebra of type $\mathtt{B}$. To do that, previously we introduce a framed version of both, the braid group and the Coxeter group of type ${\mathtt B}$.
At the end of the section we include  some useful relations derived directly from the defining relations of our framization algebra.

\subsection{}

We start with the definition of a $d$--framed version of $W_n$.

 \begin{definition}\rm
 The {\it $d$--modular framed Coxeter group of type ${\mathtt B_n}$}, $W_{d,n}$, is defined as the group generated by $\R_1 , \s_1, \ldots, \s_{n-1}$ and $t_1,\ldots , t_n$ satisfying the Coxeter relations of type ${\mathtt B}_n$ among  $\R_1$ and the $\s_i$'s, the relations  $t_it_j = t_jt_i$ for all $i,j$, the relations $t_j^d = 1$  for all $j$, together with the following relations:
\begin{equation}\label{wdn}
  \begin{array}{rcl}
 t_j\R_1 & =  &\R_1 t_j\qquad \text{for all j}\\
 t_j\s_i  & =  &\s_i t_{\s_i(j)}.
\end{array}
\end{equation}

The analogous group defined by the same presentation, where only relations $t_j^d = 1$ are omitted, shall be called {\it framed Coxeter group of type ${\mathtt B_n}$} and will be denoted as $W_{\infty,n}$.
 \end{definition}

\begin{definition}\rm
The {\it framed braid group of type  ${\mathtt B}_n$}, denoted $\mathcal{F}_{n}^{\mathtt{B}}$, is the group presented by generators $\rho_1, \sigma_1, \ldots ,\sigma_{n-1}$, $t_1, \ldots , t_n$
subject to the  relations (\ref{fbraidA}) and (\ref{braidB}),
together with the following relations:
\begin{equation}\label{fbraidrB}
\begin{array}{ccl}
t_i \rho_1 & = & \rho_1 t_i \quad \text{for all i}.\\
  \end{array}
\end{equation}

The {\it $d$--modular framed braid group}, denoted  $\mathcal{F}_{d,n}^{\mathtt{B}}$, is  defined as the group obtained  by adding the relations $t_i^d=1$, for all $i$, to the above defining presentation of $\mathcal{F}_{n}^{\mathtt{B}}$.
\end{definition}

In Figure~\ref{braidgenrs} the generators of the groups  $\mathcal{F}_{n}^{\mathtt{B}}$ and $\mathcal{F}_{d,n}^{\mathtt{B}}$ are illustrated.

\begin{figure}[h]
\centering
  \includegraphics{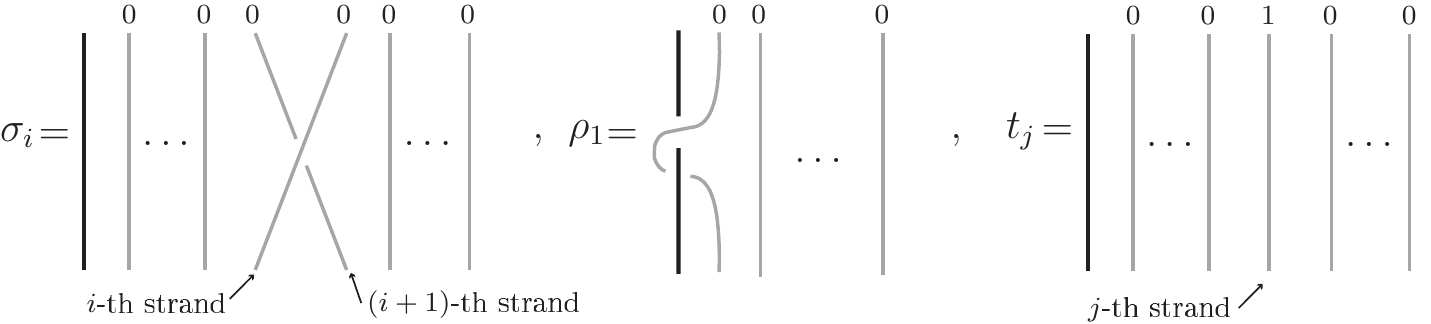}
	\caption{The generators of the groups $\mathcal{F}_{n}^{\mathtt{B}}$ and $\mathcal{F}_{d,n}^{\mathtt{B}}$.\label{braidgenrs}}

 \end{figure}

The mapping that acts as the identity on the generators $\R_1$ and the $\s_i$'s and maps the  $t_j$'s to  $1$ defines a morphism from $W_{d,n}$ onto $W_n$.  Also, we have the  natural epimorphism  from $\mathcal{F}_{d,n}^{\mathtt{B}}$ onto $W_{d,n}$ defined  as the identity on the $t_j$'s and mapping $\rho_1$ to $\R_1$ and  $\sigma_i$ to $\s_i$, for all $i$. Thus, we have the following sequence of epimorphisms.
$$
\mathcal{F}_n^{\mathtt{B}}\longrightarrow\mathcal{F}_{d,n}^{\mathtt{B}}\longrightarrow W_{d,n}
\longrightarrow W_n
$$
where the first arrow is the natural projection of $\mathcal{F}_n^{\mathtt{B}}$ to
 $\mathcal{F}_{d,n}^{\mathtt{B}}$.
\smallbreak
Finally, we can lift trivially the length function $\ell$ on $W_n$ to $W_{d,n}$. Indeed, we deduce  from relations (\ref{wdn}), that every $x\in W_{d, n}$ can be written in the form $x= wt_1^{a_1}\cdots t_n^{a_n}$, with $w\in W_n$. Thus, we define the length of $x$  by $\ell(w)$. We denote again by $\ell$ the length function on $W_{d,n}$.

\smallbreak
Geometrically, elements in $\mathcal{F}_n$ (resp. $\mathcal{F}_{d,n}$) are braids of type ${\mathtt{B}}$ (recall Figure~\ref{bbraid}) such that each one of the $n$ moving strands has a framing in ${\mathbb Z}$ (resp. ${\mathbb Z}/d{\mathbb Z}$) attached.

\begin{remark}[Symmetric framed braids]\label{symfrbraids} \rm
Following Remark~\ref{symbraids1}, one can define analogously the \textit{symmetric framed braids} resp. \textit{$d$-modular symmetric framed braids} (denoted by $Z\mathcal{F}_n$ resp. $Z\mathcal{F}_{d,n}$) and prove that the liftings $\widetilde{W}_{\infty,n}$ resp. $\widetilde{W}_{d,n}$ are isomorphic to these groups of braids.
 In terms of geometric relizations, a symmetric framed braid is a braid in $ZB_n$ with an  integer resp. $d$-modular framing attached on each strand, such that the strands $i$ and $-i$ have the same framing, for all $i$.
\end{remark}

Note  that $W_{n}$  can be regarded naturally as a subgroup of $W_{d,n}$; indeed, the elements of $W_n$ correspond to the elements of $W_{d,n}$ having all framings equal to $0$.
We will proceed now to lift the sets $\mathtt{N}_k$ of $W_n$, introduced in Subsection \ref{coxeterB}, to subsets  $\mathtt{N}_{d,k}$ of the group $W_{d,n}$, with the aim to give a standard writing for the elements of $W_{d,n}$ which will be useful to parameterize later a basis of the framization of the Hecke algebra of type $\mathtt{B}_n$ defined here.

\smallbreak
We define inductively the  subsets $\mathtt{N}_{d,k}$ of  $W_{d,n}$ as follows:
$$
\mathtt{N}_{d,1}:=\{t_1^m,\R_1 t_1^m ;\ 0\leq m \leq d-1 \}
$$
 and
$$
\mathtt{N}_{d,k}=\{t_k^m, \R_{k}t_k^{m}, \s_{k-1}x\ ;\ x\in N_{d,k-1}, 0\leq m \leq d-1 \}
\quad \mbox{for all $2\leq k\leq n$.}
$$
Note that, for all $k$ and $d$, we have  $\mathtt{N}_k\subseteq \mathtt{N}_{d,k}$ and for $d=1$ the sets $\mathtt{N}_k$ and $\mathtt{N}_{1,k}$ coincide. Also,   every element $x\in \mathtt{N}_{d,k}$ can be written as $x=y t_k^m$, with $y\in \mathtt{N}_k$. Further, we have the following proposition.

\begin{proposition}\label{standardWdn}
Every element of $W_{d,n}$ can be expressed in the standard form, that is, as a product  $m_1\ldots m_n$, where  $m_i\in \mathtt{N}_{d,i}$.
\end{proposition}
\begin{proof}
Set $x= w_1\ldots w_nt_1^{a_1}\ldots t_n^{a_n}\in W_{d,n}$. We will prove the claim by induction on $n$.
 For $n=2$, it is straightforward to check that $x$ can be written in standard form. E.g. if $w_1 = \R_1$ and $w_2 = s_1\R_1$, we have
 $$
x= w_1w_2t_1^at_2^b = (\R_1)(s_1\R_1) t_1^at_2^b =(\R_1t_1^b)(s_1\R_1 t_1^a)= m_1m_2
 $$
where $m_1= \R_1t_1^b$ and $m_2= s_1\R_1 t_1^a$. Suppose now that the proposition is true for all positive integers less than $n$. Then, if  $w_n$ in $x$ is equal to $1$ or $\R_n$, we have:
$$
x= (w_1\ldots w_{n-1}t_1^{a_1}\ldots t_{n-1}^{a_{n-1}})(w_nt_n^{a_n}).
$$
Now, $w_nt_n^{a_n}\in \mathtt{N}_{d,n}$ and applying the induction hypothesis on the word inside the first parenthesis above, we deduce that $x$ can be written in the standard form.

If $w_n$ is equal to $ s_{n-1}\ldots s_i$ or $ s_{n-1}\ldots s_i\R_i$, we have:
\begin{eqnarray*}
x & =  &  w_1\ldots w_{n-1}(w_nt_i^{a_i})t_1^{a_1} \ldots t_{i-1}^{a_{i-1}}t_{i+1}^{a_{i+1}}\ldots t_n^{a_n}\\
& = & (w_1 \ldots w_{n-1}t_1^{a_1}\ldots t_{i-1}^{a_{i-1}}t_i^{a_{i+1}}\ldots t_{n-1}^{a_n})(w_nt_i^{a_i}).
\end{eqnarray*}
Again, noting that $w_nt_i^{a_i}\in \mathtt{N}_{d,n}$ and applying the induction hypothesis on the word inside the first parenthesis above, it follows that $x$ can be written in the standard form.
\end{proof}

\subsection{}

In order to define a framization of the Hecke algebra of type $\mathtt{B}$,  we need to introduce the following elements $f_1$ and $e_i$, for $i= 1, \ldots, n-1$, in  ${\Bbb K}[{\mathcal F}_{d,n}^{\mathtt B}]$,
$$
f_1:= \frac{1}{d}\sum_{m=0}^{d-1} t_1^{m}
 \quad \text{and}\quad
e_i:= \frac{1}{d}\sum_{m=0}^{d-1} t_i^{m}t_{i+1}^{d-m}\quad \text{for all}\quad 1\leq i\leq n-1.
$$
Notice that the $f_1$ and the $e_i$'s are idempotents, cf \cite{jula5} for the $e_i$'s.
\begin{definition} \rm
Let $n\geq 2$.
 The algebra ${\rm Y}_{d, n}^{\mathtt{B}} = {\rm Y}_{d, n}^{\mathtt{B}}(\U, \V)$  is defined  as the quotient of
 ${\Bbb K}[\mathcal{F}_{d,n}^{\mathtt{B}}]$ over the two--sided ideal generated by the following elements:
$$
\rho_1^2 - 1 - (\V - \V^{-1})f_1\rho_1 \quad \text{and} \quad    \sigma_i^2 - 1 -(\U - \U^{-1})e_i\sigma_i\quad \text{for all}\quad 1\leq i\leq n-1.
$$
\end{definition}

We shall denote  the corresponding to $\sigma_i$ (respectively, to $\rho_1$) in $ {\rm Y}_{d, n}^{\mathtt{B}}$ by $g_i$ (respectively, by $b_1$) and we shall keep the same   notation for the $t_j$'s  (respectively, the $e_i$'s and $f_1$) in  ${\rm Y}_{d, n}^{\mathtt{B}}$.
Hence, equivalently, the algebra $\Y$ can be defined by generators $1, b_1, g_1, \ldots , g_{n-1}$, $t_1, \ldots , t_{n}$ and relations as follows.

\begin{eqnarray}g_ig_j & = & g_jg_i  \quad \text{ for} \quad \vert i-j\vert > 1\label{braid1}\\
  g_i g_j g_i & = & g_j g_i g_j \quad \text{ for} \quad \vert i-j\vert = 1 \label{braid2}\\
  b_1 g_i & = & g_i b_1 \quad \text{for all}\quad  i\not= 1 \label{braid3}\\
b_1 g_1 b_1 g_1 & = &  g_1 b_1 g_1 b_1 \label{braid4}\\
t_i t_j & =  & t_j t_i  \quad \text{for all }  \  i, j\label{modular2}\\
  t_j g_i & =  & g_i t_{s_i(j)}\quad \text{for all }\,  i, j  \label{th}\\
t_i b_1 & = & b_1 t_i \quad \text{for all i}\quad   \label{tb1}\\
t_i^d & = & 1 \quad \text{for all}\quad i \label{modular1}\\
g_i^2 & = &  1+ (\U-\U^{-1})e_ig_i\quad \text{for all}\quad 1\leq i\leq n-1	\label{quadraticU}\\
b_1^2 & = &  1 + (\V-\V^{-1})f_1b_1.\label{quadraticV}
\end{eqnarray}

\begin{note}\rm
We extend the above definition for $n=1$ by defining   $\Y$ as the algebra generated by $t_1$ and $b_1$ subject to the relations (\ref{tb1}), (\ref{modular1}) and (\ref{quadraticV}).
\end{note}

\begin{remark}\label{diff}\rm
 Note that $\Y$ is different from the algebra ${\rm Y}(d,m,n)$, for $m=2$ and suitable parameters $\lambda_1$ and $\lambda_2$, defined by M. Chlouveraki and L. Poulain d'Andecy in \cite{chpoIMRN}. Indeed, they differ in the quadratic relation of the generator $b_1$, since in $\Y$ the relation (\ref{quadraticV}) involves framing generators, meanwhile the quadratic relation defined in ${\rm Y}(d,2,n)$ doesn't.
\end{remark}

From the above description by generators and relations of the algebras  $\Y$ we have  $\Y\subseteq {\rm Y}_{d, n+1}^{\mathtt{B}}$, for all $n\geq 1$. Thus, by taking ${\rm Y}_{d,0}^{\mathtt{B}}:= {\Bbb K}$, we have
that following tower of algebras.
\begin{equation}\label{tower}
{\rm Y}_{d,0}^{\mathtt{B}}\subseteq {\rm Y}_{d,1}^{\mathtt{B}}\subseteq \cdots \subseteq{\rm Y}_{d,n}^{\mathtt{B}} \subseteq {\rm Y}_{d,n+1}^{\mathtt{B}}\subseteq \cdots
\end{equation}

It is clear that  $f_1$ commutes with $b_1$  and  $e_i$ commutes with $g_i$. These facts implies that the generators $b_1$ and $g_i$'s are invertible. Moreover, we have:
\begin{equation}\label{inverse}
b_1^{-1} = b_1 - (\V-\V^{-1})f_1\quad \text{and}\quad g_i^{-1} = g_i - (\U -\U^{-1})e_i.
\end{equation}

\begin{remark}\rm
Notice that, by taking $d=1$, the algebra ${\rm Y}_{1,n}^{\mathtt B}$ becomes ${\rm H}_n(\U,\V)$. Further, by mapping $g_i\mapsto h_i$ and $t_i\mapsto 1$, we obtain an epimorphism from $\Y$ to ${\rm H}_n(\U,\V)$. Moreover, if we map the $t_i$'s  to a fixed non--trivial $d$--th root of the unity, we have an epimorphism from $\Y$ to ${\rm H}_n(\U,1)$.
\end{remark}

\begin{remark}\rm
Notice that the relations (\ref{braid1})--(\ref{tb1}) are the defining relations of  $\mathcal{F}_n^{\mathtt{B}}$ and the relations (\ref{braid1})--(\ref{modular1}) are the defining relation of $\mathcal{F}_{d,n}^{\mathtt{B}}$. Then,  ${\rm Y}_{d, n}^{\mathtt{B}}$ can be obviously defined  as  a quotient of the group algebra $\mathbb{K}[\mathcal{F}_n^{\mathtt{B}}]$ or $\mathbb{K}[\mathcal{F}_{d,n}^{\mathtt{B}}]$. Alternatively, ${\rm Y}_{d, n}^{\mathtt{B}}$ can be regarded as  a
$(\U, \V)$--deformation of the group algebra ${\Bbb K}[W_{d,n}]$ of the $d$-modular framed braid group of type $\mathtt{B}$.
\end{remark}

\subsection{}

We also have the following relations which are deduced easily and will be  used frequently in the sequel.
\begin{eqnarray*}
e_i b_1 &  = &   b_1e_i \quad \text{for all} \ i\\
f_j g_i & = &   g_i f_j \quad \text{for}\quad \vert i-j\vert >1
\end{eqnarray*}
where $f_j$ is the natural generalization of $f_1$,
$$
f_j :=\frac{1}{d}\sum_{m=0 }^{d-1}t_j^m.
$$
Notice that the $f_j$'s are idempotents.

Finally, we finish the section by introducing certain elements $b_i\in \Y$ and  proving some algebraic identities   which will be used along the paper.
 Set
$$
b_i := g_{i-1}\ldots g_1 b_1 g_1^{-1}\ldots g_{i-1}^{-1}\quad   \text{for all}\quad 2\leq i\leq n.$$

\newpage
Figure~\ref{bi} illustrates the elements $b_i$.

\begin{figure}[h]
\begin{center}
  \includegraphics{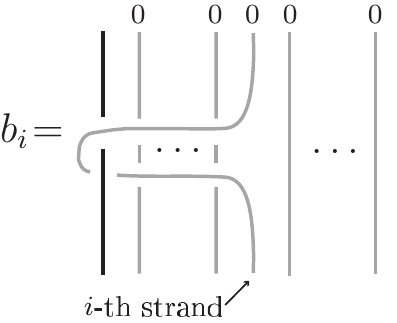}
	\caption{The elements $b_i$.}\label{bi}
	\end{center}
 \end{figure}

  Further, for all $i$ we have  $b_i f_i = f_i b_i$ and a direct computation shows that,
$$b_i^2 = 1 + (\V-\V^{-1})f_i b_i\quad\mbox{and}\quad b_i^{-1} = b_i - (\V -\V^{-1})f_i.$$

\begin{proposition}\label{braidHBHB}
For $n\geq 2$ and $1\leq i,k\leq n-1$, the following relations hold in $\Y$:
\begin{enumerate}
\item[(i)] $b_kt_j=t_jb_k$, for all j
\item[(ii)] $b_k g_i = g_i b_k$, for $i\leq k-2$ or $i\geq k+1$
\item[(iii)] $g_k b_{k} g_k b_{k} =  b_{k} g_k b_{k}g_k $
\item[(iv)] $g_kb_{k}b_{k+1} = b_{k}g_kb_{k}$.
\end{enumerate}
\end{proposition}
\begin{proof}
The proof of relations (i) and (ii) follows directly by using the defining relations of  $\Y$. The proof of  relations (iii) is by induction on $k$. For $k=1$, the relation in question is the defining relation (\ref{braid4}). Let us suppose now that relation (iii) holds for all positive integers less than $k+1$. Then for $k+1$ we have:
\begin{eqnarray*}
g_{k+1} b_{k+1} g_{k+1} b_{k+1}&=&g_{k+1}(g_kb_{k}g_k^{-1}) g_{k+1} (g_kb_{k}g_k^{-1}) \\
&\stackrel{(\ref{braid2})}{=}& g_{k+1} g_kb_{k}g_{k+1} g_k g_{k+1}^{-1}b_{k}g_k^{-1} \\
& = & g_{k+1} g_k g_{k+1}b_{k} g_k b_{k}g_{k+1}^{-1}g_k^{-1}\\
& \stackrel{(\ref{braid2})}{=}& g_kg_{k+1} g_kb_{k} g_k b_{k}g_{k+1}^{-1}g_k^{-1}\\
& \stackrel{({\rm ind.})}{=} & g_k g_{k+1}b_{k} g_k b_{k}g_k g_{k+1}^{-1}g_k^{-1}\\
& = & g_k b_{k}g_{k+1} g_k b_{k}g_k g_{k+1}^{-1} g_k^{-1}\\
& = & b_{k+1}g_k g_{k+1} g_k b_{k}g_{k}g_{k+1}^{-1}g_k^{-1}\\
& \stackrel{(\ref{braid2})}{=}& b_{k+1}g_{k+1} g_k g_{k+1}  b_{k}g_k g_{k+1}^{-1} g_k^{-1}\\
& =& b_{k+1}g_{k+1} g_k  b_{k}g_{k+1}g_kg_{k+1}^{-1}g_k^{-1}\\
& = & b_{k+1}g_{k+1}  b_{k+1} g_kg_{k+1}g_kg_{k+1}^{-1}g_k^{-1}\\
& \stackrel{(\ref{braid2})}{=} & b_{k+1}g_{k+1}  b_{k+1} g_{k+1}g_kg_{k+1}g_{k+1}^{-1}g_k^{-1} \\
&=&  b_{k+1}g_{k+1} b_{k+1} g_{k+1}.
\end{eqnarray*}

We prove now relation (iv). We have $g_k b_{k} b_{k+1} = g_kb_{k}(g_kb_{k}g_k^{-1}) =  (g_kb_{k}g_kb_{k})g_k^{-1}$. Then by using relation (iii) we obtain
$g_k b_{k} b_{k+1} = (b_{k}g_kb_{k}g_k) g_k^{-1} = b_{k}g_kb_{k}$, so relation (iv) is true.
\end{proof}

\section{A tensorial representation for $\Y$}

We will define now  a tensorial representation for the algebra $\Y$.
The definition of this representation is based on the tensorial representation constructed by Green in \cite{gr} for the Hecke algebra of type $\mathtt B$ and  following the idea of an  extension of the Jimbo representation of the Hecke algebra of type $\mathtt A$ to the Yokonuma-Hecke algebra proposed by Espinoza and Ryom--Hansen in \cite{esry}. The tensorial representation constructed here will be used to prove that a set of linear generators, denoted ${\mathsf D}_n$, is a basis for $\Y$ (Theorem \ref{basis}). Further, as a corollary, we obtain that this tensorial representation is faithful (Corollary \ref{faithful}).

\subsection{}

Let $V$ be a  ${\Bbb K}$--vector space with basis $\mathcal{B}=\{v_i^r\, ;\,  i\in X_n,\, 0\leq r\leq d-1\} $.  As usual we denote by $\mathcal{B}^{\otimes k}$ the natural  basis
of  $V^{\otimes k}$ associated to $\mathcal{B}$. That is, the elements of $\mathcal{B}^{\otimes k}$  are of the form:
$$
v_{i_1}^{m_1}\otimes\cdots\otimes v_{i_k}^{m_k}
$$
where $(i_1, \ldots , i_k)\in X_n^k$ and $(m_1, \ldots , m_k)\in ({\Bbb Z}/d{\Bbb Z})^k$.

 We define the endomorphism $T$ of $V$ by
$$
(v_i^r )T=\omega^r v_i^r
$$
and the endomorphism $G$ of $V\otimes V$ by
$$
(\vnor)G=\left\{\begin{array}{cl}
\U\vinv & \text{ for}\quad i=j\  \text{and } \  r=s\\
\vinv & \text{ for}\quad i< j\  \text{and } \  r=s\\
\vinv+ (\U-\U^{-1})\vnor & \text{ for}\quad i>j\ \text{and }\  r=s\\
\vinv & \text{ for}\quad  r\not=s.
\end{array} \right.
$$
For all $1\leq i \leq n-1$, we extend these endomorphisms to the endomorphisms $T_i$ and $G_i$ of the $n$th tensor power $V^{\otimes n}$ of $V$, as follows:
$$
T_i:=1_V^{\otimes(i-1)}\otimes T\otimes  1_V^{\otimes(n-i)} \qquad\text{and}\qquad
G_i:= 1_V^{\otimes(i-1)}\otimes G\otimes 1_V^{\otimes(n-i-1)}
$$
where $1_V^{\otimes k}$ denotes the endomorphism identity of $V^{\otimes k}$. Further we define the endomorphism $B_1$ of $V^{\otimes n}$ by:
$$
(v_{i_1}^{r_1}\otimes \dots \otimes v_{i_n}^{r_n} )B_1=
\left\{\begin{array}{cl}
 v_{-i_1}^{r_1}\otimes \dots \otimes v_{i_n}^{r_n} & \text{for}\quad i_1>0\ \text{and}\ r_1=0\\
 v_{-i_1}^{r_1}\otimes \dots \otimes v_{i_n}^{r_n}+(\V-\V^{-1})v_{i_1}^{r_1}\otimes \dots \otimes v_{i_n}^{r_n}  &\text{for}\quad i_1<0\ \text{and}\ r_1=0\\
  v_{-i_1}^{r_1}\otimes \dots \otimes v_{i_n}^{r_n} & \text{for}\quad r_1\not=0.
\end{array} \right.
$$

The main goal of this section is to prove that these endomorphisms define a representation of the $\Y$ in the algebra of endomorphisms ${\rm End}(V^{\otimes n})$ of $V^{\otimes n}$. To do that, we will need  certain  endomorphisms $E_i$ of $V^{\otimes n}$, introduced in \cite{esry}, which are defined by,
$$
E_i = \frac{1}{d}\sum_{m=0}^{d-1}T_i^mT_{i+1}^{-m} \qquad (1 \leq i \leq n-1).
$$
Also, we will need  the following element $F \in {\rm End}(V^{\otimes n})$,
$$
F := \frac{1}{d}\sum_{m=0}^{d-1}T_1^m.
$$

\begin{lemma}\label{F}
We have:
\begin{enumerate}
\item
$$
(v_i^r\otimes v_j^s)E_i=\left\{\begin{array}{cc}
 0 & r\not=s\\
 v_i^r\otimes v_j^s & r=s
\end{array} \right.
$$
\item
$$
(v_{i_1}^{r_1}\otimes \dots \otimes v_{i_n}^{r_n})F=\left\{\begin{array}{cc}
 0 & r_1>0\\
 v_{i_1}^{r_1}\otimes \dots \otimes v_{i_n}^{r_n} & r_1=0.
\end{array} \right.
$$
\end{enumerate}
\end{lemma}

\begin{proof}
Claim (1) is \cite[Lemma 3]{esry}. To prove (2), we note that $(v_{i_1}^{r_1}\otimes \dots \otimes v_{i_n}^{r_n})T_1^m = \omega^{mr_1}v_{i_1}^{r_1}\otimes \dots \otimes v_{i_n}^{r_n}$. Hence,
$$
(v_{i_1}^{r_1}\otimes \dots \otimes v_{i_n}^{r_n})F=\left(\frac{1}{d}\sum_{m=0}^{d-1}\omega^{mr_1}\right)v_{i_1}^{r_1}\otimes \dots \otimes v_{i_n}^{r_n}.
$$
Now, from the fact that  $\frac{1}{d}\sum_{m=0}^{d-1}\omega^{mr_1}$ is 1 or 0, depending if $r_1=0$ or not,  claim (2) follows.
\end{proof}

\begin{theorem}\label{tensorrepre}
The mapping   $b_1\mapsto B_1$, $g_i\mapsto G_i$ and $t_i\mapsto T_i$ defines a representation $\Phi$ of $\Y$  in ${\rm End}(V^{\otimes n})$.
\end{theorem}
\begin{proof}
To prove the theorem it is enough to verify that the operators $G_i$, $T_i$ and  $B_1$ satisfy the defining relations of $\Y$ whenever we  replace $g_i$ by $G_i$, $t_i$ by $T_i$ and  $b_1$ by $B_1$.
Having present \cite[Theorem 4]{esry} follows that (\ref{braid1}), (\ref{braid2}), (\ref{th})--(\ref{modular1}) and (\ref{quadraticU}) it holds for the operators  $G_i$, $T_i$, and $B_1$, and it is easy to check  that the identities (\ref{braid3}) and (\ref{tb1}) hold for these operators too. In order to finish the proof of the theorem we shall prove the relations (\ref{quadraticV}) and (\ref{braid4}). To do that,  it  is enough to  assume   $n=2$ and to prove the relations by evaluating in  $x:=v_i^r\otimes v_j^s$.
\smallbreak
We prove now that relation (\ref{quadraticV}) is valid  for $B_1$. If in $x$ we have $r=0$, we distinguish either $i>0$ or $i<0$. For
$i>0$; we have:
$$
(x)B_1^2=(v_{-i}^{r}\otimes v_j^s)B_1=v_{i}^r\otimes v_j^s+(\V-\V^{-1})v_{-i}^{r}\otimes v_j^s=(x)[1_V^{\otimes 2}+(\V-\V^{-1})B_1].
$$
For $i<0$, we have:
$$
(x)B_1^2
= (v_{-i}^{r}\otimes v_j^s +(\V-\V^{-1})v_i^r \otimes v_j^s) B_1
 =
v_{i}^{r}\otimes v_j^s+(\V-\V^{-1})(v_i^r\otimes v_j^s)B_1
=
(x)[1_V^{\otimes 2}+(\V-\V^{-1})B_1].
$$
On the other hand, if $r\not = 0$ it is clear that $B_1^2=1_V^{\otimes 2}$. Then by Lemma \ref{F} relation (\ref{tb1}) holds in both cases.
\smallbreak
Finally,  we will prove that the identity (\ref{braid4}) holds if we replace $b_1$ by  $B_1$ and $g_1$ by  $G_1$. To do that, we  will distinguish first the cases according to the following exhaustive  values of $r$ and $s$:
\begin{itemize}
\item[(a)] Case $r=s=0$
\item[(b)] Case $r\not=0$ and $r\not=s$
\item[(c)] Case $r=0$ and $s\not=0$
\item[(d)] Case $r\not=0$ and $r=s$.
\end{itemize}
Case (a) holds by \cite{gr}, and case (b) is easy to check. We distinguish now according the values of $i$ and $j$ in the remaining cases (c) and (d). For item (c) we have four cases and eight cases for item (d). We will check only the most representative cases by evaluating in $x=v_i^r \otimes v_j^s$.

\textbf{Case:} $r=0$, $ s\not=0$, $i>0$ and  $j<0 $. We have:
\begin{eqnarray*}
(x)G_1 B_1 G_1B_1 &=& (\vinv)B_1 G_1 B_1 \\
 &=&(\vinvn + (\V - \V^{-1})\vinv )G_1 B_1 \\
 &=&(v_i^r\otimes v_{-j}^{s}+ (\V - \V^{-1})\vnor)B_1\\
 &=& v_{-i}^{r}\otimes v_{-j}^{s} + (\V - \V^{-1})\vnorn\\
 &=&(v_{-j}^{s}\otimes v_{-i}^{r} +(\V -\V^{-1})v_{j}^{s}\otimes v_{-i}^{r} )G_1 \\
 &=&( v_{j}^{s}\otimes v_{-i}^{r} ) B_1 G_1 = (\vnorn)G_1B_1G_1 = (x)B_1 G_1 B_1 G_1.
\end{eqnarray*}

\textbf{Case:} $r=0$, $s\not=0$, $i<0$ and $j<0$. We have:

\begin{eqnarray*}
(x) G_1 B_1 G_1B_1&=& (\vinv)B_1 G_1 B_1 \\
 &=&(\vinvn + (\V -\V^{-1})\vinv )G_1 B_1 \\
 &=&(v_i^r\otimes v_{-j}^{s}+ (\V -\V^{-1})\vnor)B_1 \\
 &=&v_{-i}^{r}\otimes v_{-j}^{s} + (\V -\V^{-1})v_i^r\otimes v_{-j}^{s}+ (\V -\V^{-1})\vnorn + (\V -\V^{-1})^2\vnor\\
 &=&(v_{-j}^{s}\otimes v_{-i}^{r}+(\V -\V^{-1})v_{j}^{s}\otimes v_{-i}^{r} +(\V -\V^{-1})\vinvn \\
 & & + (\V -\V^{-1})^2\vinv)G_1\\
 &=&(v_{j}^{s}\otimes v_{-i}^{r}+ (\V -\V^{-1})\vinv)B_1 G_1\\
 &=& (\vnorn+ (\V -\V^{-1})\vnor)G_1B_1G_1 = (x)B_1 G_1 B_1 G_1.
\end{eqnarray*}

\textbf{Case:} $r\not=0$, $s=r$, $i>j$ and $-j< i$. We have:
\begin{eqnarray*}
(x) G_1 B_1 G_1B_1&=& (\vinv+(\U -\U^{-1})\vnor) B_1 G_1 B_1 \\
 &=&(\vinvn + (\U -\U^{-1})\vnorn )G_1 B_1 \\
 &=&( v_{i}^{r}\otimes v_{-j}^{s}+ (\U -\U^{-1})v_{j}^{s}\otimes v_{-i}^{r})B_1 \\
 &=& v_{-i}^{r}\otimes v_{-j}^{s}+ (\U -\U^{-1})v_{-j}^{s}\otimes v_{-i}^{r} \\
 &=&( v_{-j}^{s}\otimes v_{-i}^{r})G_1 =
(v_{j}^{s}\otimes v_{-i}^{r}) B_1 G_1 =
 (\vnorn)G_1B_1G_1 \\
 & =& (x)B_1 G_1 B_1 G_1.
\end{eqnarray*}

\textbf{Case:} $r\not=0$, $s=r$, $i>j$ and $-j> i$. We have:
\begin{eqnarray*}
(x)G_1 B_1 G_1B_1&=& (\vinv+(\U -1)\vnor) B_1 G_1 B_1 \\
 &=&(\vinvn + (\U -\U^{-1})\vnorn )G_1 B_1 \\
 &=&(v_{i}^{r}\otimes v_{-j}^{s}+ (\U -\U^{-1})\vinvn + (\U -\U^{-1})v_{j}^{s}\otimes v_{-i}^{r} \\
 &  & + (\U -\U^{-1})^2 \vnorn)B_1 \\
 &=&v_{-i}^{r}\otimes v_{-j}^{s}+ (\U -\U^{-1})\vinv + (\U -\U^{-1})v_{-j}^{s}\otimes v_{-i}^{r}+ (\U -\U^{-1})^2 \vnor\\
 &=&(v_{-j}^{s}\otimes v_{-i}^{r}+(\U -\U^{-1})\vnor)G_1\\
 &=&(v_{j}^{s}\otimes v_{-i}^{r}+(\U -\U^{-1})\vnorn)B_1 G_1\\
 & = & (\vnorn)G_1B_1G_1 =(x)B_1 G_1 B_1 G_1.
\end{eqnarray*}

\textbf{Case:} $r\not=0$, $s=r$, $ i>j$ and  $-j=i$. We have:
\begin{eqnarray*}
(x)G_1 B_1 G_1B_1&=& (\vinv+(\U -\U^{-1})\vnor)B_1 G_1 B_1 \\
 &=&(\vinvn + (\U -\U^{-1})\vnorn )G_1 B_1 \\
 &=&(\U v_{i}^{r}\otimes v_{-j}^{s}+ \U (\U -\U^{-1})v_{j}^{s}\otimes v_{-i}^{r})B_1 \\
 &=& \U v_{-i}^{r}\otimes v_{-j}^{s}+ \U (\U -\U^{-1})v_{-j}^{s}\otimes v_{-i}^{r}\\
 &=&( \U v_{-j}^{s}\otimes v_{-i}^{r})G_1 =
(\U v_{j}^{s}\otimes v_{-i}^{r})B_1 G_1 =
(\vnorn) G_1B_1G_1\\
& =& (x)B_1 G_1 B_1 G_1.
\end{eqnarray*}
\end{proof}

\subsection{}

We shall finish the section by proving Proposition \ref{aplicacion}, which is an analogue of \cite[Lemma 3.1.4]{gr}. This proposition will be used in the proof of Theorem \ref{basis} and describes, through $\Phi$, the action of $W_n$ on the basis $\mathcal{B}^{\otimes n}$.
\smallbreak
The defining generators $b_1$ and  $g_i$ of the algebra $\Y$ satisfy the same braid relations as the Coxeter  generators  $\R$ and $\s_i$ of the group $W_n$. Thus, the well--known  Matsumoto Lemma  implies that  if $w_1\ldots w_m$ is a reduced expression of $w\in W_n$, with $w_i\in \{\R, \s_1, \ldots, \s_{n-1}\}$, then the following element $g_w$ is well--defined:
\begin{equation}\label{gw}
g_w := g_{w_1}\cdots  g_{w_m}
\end{equation}
where $g_{w_i} = b_1$, if $w_i=\R$ and $g_{w_i} = g_j$, if  $w_i = \s_j$.

\smallbreak

The notation $\Phi_w$ stands  for the image by $\Phi$ of $g_w\in \Y$. Note that, for $w$, $w^{\prime}\in W_n$ such that $\ell(ww^{\prime}) = \ell(w)+ \ell(w^{\prime})$, we have $\Phi_{ww^{\prime}}= \Phi_w \Phi_{w{^{\prime}}}$.

\begin{proposition}\label{aplicacion}
Let $w\in W_n$ parameterized by $(m_1,\dots,m_n)\in X_n^n$. Then
$$
(v_1^{r_1}\otimes \dots \otimes v_n^{r_n})\Phi_{w}=v_{m_1}^{r_{\vert m_1\vert}}\otimes \dots \otimes v_{m_n}^{r_{\vert m_n\vert}}.
$$
\end{proposition}
\begin{proof}
The proof follows by induction on the length of $w$. For $l(w)=1$ we have that $w\in \{\R_1,\s_1,\dots, \s_{n-1}\}$,
then the result is direct from the definition of  $\Phi$. Now suppose that the induction hypothesis holds for any $w'\in W_n$ with $l(w')=n-1$ and let $w$ be an element with length $n$. Then we have two cases: $w=w'\mathtt{r}$ or $w=w'\mathtt{s}_i$ for some $w'\in W_n$ with $l(w')=n-1$.  We only present the proof of the case  $w=w'\mathtt{s}_i$, as the proof of the other case  is analogous. Suppose  $w'$ is parameterized by $(m_1,\dots, m_n)\in X_n^n$. Then  by the induction hypothesis we obtain:
$$
(v_1^{r_1}\otimes \dots \otimes v_n^{r_n})\Phi_{w}=(v_1^{r_1}\otimes \dots \otimes v_n^{r_n})\Phi_{w'}G_i=(v_{m_1}^{r_{\vert m_1\vert}}\otimes \dots \otimes v_{m_n}^{r_{\vert m_n\vert}})G_i.
$$
Now, in Lemma \ref{coxeter} we have $m_i<m_{i+1}$, therefore from the definition of $G_i$'s we obtain
$$
(v_1^{r_1}\otimes \dots \otimes v_n^{r_n})\Phi_{w}=v_{m_1}^{r_{\vert m_1\vert}}\otimes \dots \otimes v_{m_{i+1}}^{r_{\vert m_{i+1}\vert}}\otimes v_{m_i}^{r_{\vert m_i\vert}} \otimes\dots v_{m_n}^{r_{\vert m_n\vert}}.
$$
Finally, from (\ref{parametrized}) we have that $w$ is parameterized by
$(m_1,\ldots ,m_{i+1}, m_i,\ldots , m_n)$. Hence the claim follows.
\end{proof}

\section{Linear bases for $\Y$}

We introduce here two linear bases $\mathsf{C}_n$ and $\mathsf{D}_n$ for $\Y$. The first one is used for defining in the next section a  Markov trace on $\Y$, the second one plays a technical role for proving that  $\mathsf{C}_n$ is  a linearly independent set.

\subsection{\it The basis  $\mathsf{D}_n$}

 (Cf.\cite[Sec. 4.1]{chpoIMRN}). Set $\overline{b}_1 := b_1$, and
$$
\overline{b}_k := g_{k-1}\ldots g_1 b_1 g_1\ldots g_{k-1}\quad  \quad \mbox{for all $2\leq k\leq n$.}
$$

For all $1\leq k\leq n$, let us define inductively  the set $N_{d,k}$ by
$$
N_{d,1} := \{t_1^m, \overline{b}_1 t_1^m \,;\, 0 \leq m \leq d-1\}
$$
and
$$
N_{d,k} := \{t_k^m, \overline{b}_{k}t_k^m, g_{k-1} x\,;\, x \in N_{d,k-1},\, 0\leq m\leq d-1\}
\quad \mbox{for all $2 \leq k\leq n$.}
$$

\begin{definition}\label{Dn} \rm
We define $\mathsf{D}_n$ as the subset of $\Y$ formed by the following elements
\begin{equation}
  \mathfrak{n}_1\mathfrak{n}_2\cdots \mathfrak{n}_n
\end{equation}
where $\mathfrak{n}_i\in N_{d,i}$.
\end{definition}
We will prove first that $\mathsf{D}_n$ is a linearly spanning set for $\Y$. To do that we need some formulas of multiplication among the defining generators of $\Y$ and the elements $N_{d,k}$. These are given in Lemmas \ref{relations1D} and \ref{relations2D} below.
 Notice that  every element of $N_{d,k}$ has the form $\mathfrak{n}_{k,j,m}^+$ or  $\mathfrak{n}_{k,j,m}^-$, with $j\leq k$ and $0\leq m\leq d-1$,  where
$$
\mathfrak{n}_{k,k,m}^+:=t_k^m, \qquad \mathfrak{n}_{k,j,m}^+ := g_{k-1}\cdots g_jt_j^m\quad \text{for}\quad j<k
$$
and
$$
\mathfrak{n}_{k,k,m}^-:=\overline{b}_k t_k^m, \qquad \mathfrak{n}_{k,j,m}^- := g_{k-1}\cdots g_j\overline{b}_{j}t_j^m\quad \text{for}\quad j<k.
$$
\begin{lemma}\label{relations1D}
In $\Y$ the following relations holds:
\begin{enumerate}
  \item[(i)]
   $(\overline{b}_nt_n^{\alpha})b_1=b_1(\overline{b}_nt_n^{\alpha})$
  \item[(ii)]
   $(\overline{b}_nt_n^{\alpha})g_j=g_j(\overline{b}_nt_n^{\alpha})$, for all $j<n-1$
  \item[(iii)]
   $(\overline{b}_nt_n^{\alpha})g_{n-1}=  \mathfrak{n}_{n,n-1,\alpha}^{-}+d^{-1}(\U-\U^{-1})\sum_s t_{n-1}^{\alpha-s}\overline{b}_nt_n^{s}$
  \item[(iv)]
  $\mathfrak{n}_{n,k,\alpha}^{\pm}b_1=b_1\mathfrak{n}_{n,k,\alpha}^{\pm}$, if $k\not=1$
  \item[(v)]
   $\mathfrak{n}_{n,k,\alpha}^{+}b_1=\mathfrak{n}_{n,k,\alpha}^{-}$, if $k=1$
  \item[(vi)]
  $\mathfrak{n}_{n,k,\alpha}^{-}b_1=\mathfrak{n}_{n,k,\alpha}^{+}+d^{-1}(\V-\V^{-1})\sum_s\mathfrak{n}_{n,k,s}^{-}$, for $k=1$
  \item [(vii)]
  $\mathfrak{n}_{n,k,\alpha}^{\pm}t_j=\left\{\begin{array}{cc}
  \vspace{3mm}
                                          t_{j-1}\mathfrak{n}_{n,k,\alpha}^{\pm}  & \mbox{for $j>k$}\\
                                          \vspace{3mm}
                                           \mathfrak{n}_{n,k,\alpha+1}^{\pm}  & \mbox{for $j=k$} \\
                                           \vspace{3mm}
                                           t_j\mathfrak{n}_{n,k,\alpha}^{\pm}  & \mbox{for $j<k$.}
                                          \end{array}\right.$

  \end{enumerate}
\begin{proof}
All  relations follow from direct computations.
\end{proof}
\end{lemma}

\begin{lemma}\label{relations2D}
In $\Y$ we have:
\begin{enumerate}
  \item[(i)] $\mathfrak{n}_{n,k,\alpha}^{-}g_j=
  \left\{\begin{array}{ll}
  \vspace{3mm}
   g_{j-1}\mathfrak{n}_{n,k,\alpha}^{-}  & \mbox{for} \quad j>k\\
                                           \vspace{3mm}
                                          \mathfrak{n}_{n,k+1,\alpha}^{-} & \mbox{for}\quad j=k\\
                                          \vspace{3mm}
                                         \mathfrak{n}_{n,k-1,\alpha}^{-}+d^{-1}(\U-\U^{-1})\sum_s t_{j}^{\alpha-s}\mathfrak{n}_{n,k,s}^{-}  & \mbox{for}\quad j=k-1\\
                                         \vspace{3mm}
                                           g_j\mathfrak{n}_{n,k,\alpha}^{-}& \mbox{for}\quad j<k-1
                                          \end{array}\right.$
\vspace{3mm}
  \item[(ii)]  $\mathfrak{n}_{n,k,\alpha}^{+}g_j=
  \left\{\begin{array}{ll}
  \vspace{3mm}
                                           g_{j-1}\mathfrak{n}_{n,k,\alpha}^{+}  & \mbox{for}\quad j>k\\
                                           \vspace{3mm}
                                          \mathfrak{n}_{n,k+1,\alpha}^{+}+d^{-1}(\U-\U^{-1})\sum_s t_{j}^{\alpha-s}\mathfrak{n}_{n,k,s}^{+}& \mbox{for} \quad $j=k$ \\
                                          \vspace{3mm}
                                          \mathfrak{n}_{n,k-1,\alpha}^{+} & \mbox{for}\quad j=k-1\\
                                          \vspace{3mm}
                                           g_j\mathfrak{n}_{n,k,\alpha}^{+}& \mbox{for} \quad j<k-1.
                                          \end{array}\right.$
\end{enumerate}
\end{lemma}

\begin{proof}
In claim (i) we will check only the case $j=k-1$, since the other cases are clear. We have
  \begin{eqnarray*}
    \mathfrak{n}_{n,k,\alpha}^{-}g_{k-1} &=& g_{n-1}\dots g_1b_1g_1 \cdots g_{k-1}t_k^{\alpha} g_{k-1} \\
     &=&g_{n-1}\cdots g_1b_1g_1 \cdots g_{k-1}^2t_{k-1}^{\alpha}  \\
     &=& g_{n-1}\cdots g_1b_1g_1 \cdots g_{k-2}t_{k-1}^{\alpha}+(\U-\U^{-1})g_{n-1}\cdots g_1b_1g_1 \cdots g_{k-1}e_{k-1}t_{k-1}^{\alpha} \\
     &=&g_{n-1}\cdots g_1b_1g_1 \cdots g_{k-2}t_{k-1}^{\alpha}+d^{-1}(\U-\U^{-1})\sum_s g_{n-1}\cdots g_1b_1g_1 \cdots g_{k-1}t_{k}^st_{k-1}^{\alpha-s}  \\
     &=& \mathfrak{n}_{n,k-1,\alpha}^{-}+d^{-1}(\U-\U^{-1})\sum_s t_{k-1}^{\alpha-s}\mathfrak{n}_{n,k,s}^{-}.
  \end{eqnarray*}

The only non-trivial case in claim (ii) is whenever $j=k$. We have
\begin{eqnarray}
  \mathfrak{n}_{n,k,\alpha}^{+}g_k &=& (g_{n-1}\cdots g_k t_k^{\alpha})g_k \nonumber\\
   &=&g_{n-1}\cdots g_k^2 t_{k+1}^{\alpha}  \nonumber \\
   &=&g_{n-1}\cdots g_{k+1} t_{k+1}^{\alpha}+ (\U-\U^{-1}) g_{n-1}\cdots g_{k+1}g_k e_k t_{k+1}^{\alpha} \nonumber\\
   &=&g_{n-1}\cdots g_{k+1} t_{k+1}^{\alpha}+ d^{-1}(\U-\U^{-1})\sum_s g_{n-1}\cdots g_k t_k^s t_{k+1}^{\alpha-s} \nonumber \\
   &=&     \mathfrak{n}_{n,k+1,\alpha}^{+}+d^{-1}(\U-\U^{-1})\sum_s t_{k}^{\alpha-s}\mathfrak{n}_{n,k,s}^{+}.\nonumber
  \end{eqnarray}
\end{proof}

\begin{proposition}\label{spanD}
The set $\mathsf{D}_n$ is a  spanning set for $\Y$.
\end{proposition}
\begin{proof}
The proof
is by induction on $n$. Let $\mathfrak{D}_n$ be the  linear subspace of $\Y$ spanned by $\mathsf{D}_n$. The assertion is true for $n=1$, since $\mathsf{D}_1= \mathtt{N}_1$ and obviously ${\rm Y}_{d,1}^{\mathtt B}$ is equal to the space spanned by $N_{d,1}$.
 Assume now that $\Z$ is spanned by $\mathfrak{D}_{n-1}$. Notice that $1\in  \mathfrak{D}_n$. This fact and proving that $\mathfrak{D}_n$ is a right ideal, implies the proposition. Now, we deduce that $\mathfrak{D}_n$ is a right ideal from the hypothesis induction and  Lemmas \ref{relations1D} and \ref{relations2D}. Indeed,
 the multiplication of  $\mathfrak{n}_n\in N_{d,n}$ from the right by  all  defining generators of $\Y$ results in a linear combination of elements of the form $w\mathfrak{n}'_n$, with $\mathfrak{n}'_n\in N_{d,n}$ and $w\in \Z$.
\end{proof}

In order to prove  now that $\mathsf{D}_n$ is a linearly independent set, we will  firstly rewrite its elements in  split form, that is, as the product between the braiding part and the framing part. More precisely, given an element  in $\mathsf{D}_n$, then by using the relations (\ref{th})--(\ref{tb1}), the  framing elements (every power of the $t_j$'s) that appears in this  given element, can be moved to the right. Thus, we deduce that the elements in
$\mathsf{D}_n$ can be written in the following form:
\begin{equation}\label{elementD}
\mathfrak{r}_1\cdots \mathfrak{r}_{n}t_1^{m_1}\cdots t_n^{m_n}
\end{equation}
with $m_k\in {\Bbb Z}/d{\Bbb Z}$ and $\mathfrak{r}_k\in N_k$, where the sets $N_k$  are defined inductively as follows:
$
N_1 := \{1,  b_1 \}
$
and
$$
N_k := \{1, \overline{b}_{k}, g_{k-1} x\,;\, x \in N_{k-1}\}\quad \quad \mbox{for all $2\leq k\leq n$.}
$$
Recall now that  $g_{\s_i}= g_i$ and notice  that $\R_k$ is reduced, so $g_{\R_k}= \overline{b}_{k}$ (see (\ref{gw})). These facts and noting that the  elements of the sets $\mathtt{N}_k$ (see (\ref{NWn})) are reduced, imply that
$$
N_k =\{g_w \,;\, w\in \mathtt{N}_k\}.
$$
Then, the set ${\mathsf D}_n$ can be described  by:
$$
\mathsf{D}_n=\{g_w t_1^{m_1}\cdots t_n^{m_n}\, ;\, w\in W_n,\ (m_1,\dots, m_n)\in (\mathbb{Z}/d\mathbb{Z})^n\}.
$$

Secondly,
we shall use a certain basis $\mathcal{D}$ of $V$ introduced by Espinoza and Ryom--Hansen in \cite{esry}. More precisely, $\mathcal{D}$ consist of the following elements:
$$
u_k^r=\sum_{i=0}^{d-1}\omega^{ir}v_k^i
$$
where $k$ is running $ X_n$ and $0\leq r \leq d-1$.

Notice that  $\mathcal{D}$ is a basis for $V$, since for any fixed $k$ the base change matrix between $\{u_k^r\ ;\ 0\leq r\leq d-1\} $ and $\{v_k^s\ ; \ 0\leq s\leq d-1\}$ is non--singular, see \cite{esry}. Further, it is easy to see that
\begin{equation}\label{corre}
(u_k^r)T=u_k^{r+1}.
\end{equation}

We are now in the position to prove that $\mathsf{D}_n$ is a basis for $\Y.$
\begin{theorem}\label{basis}
$\mathsf{D}_n$ is a linear basis for $\Y$. Hence the dimension of $\Y$ is $2^nd^nn!$.
\end{theorem}

\begin{proof}
According to Proposition \ref{spanD} we only need  to  prove  that  $\mathsf{D}_n$ is a linearly independent set. Indeed, suppose that we have a linear combination in the form:
$$
\sum_{c\in \mathsf{D}_n}\lambda_cc = 0.
$$
The proof follows by proving that  $\lambda_c=0$ for all $c\in \mathsf{D}_n$. Now, using the expression (\ref{elementD}) for the elements of ${\mathsf D}_n$ and  applying $\Phi$ to the above equation,  we obtain the  following equation:
\begin{equation}\label{principal}
\sum_{m, w}\lambda_{m,w}\Phi_w T_1^{m_1}\cdots T_n^{m_n}=0
\end{equation}
where $w$ runs in $W_n$ and $m=(m_1,\dots, m_n)$ runs in $({\Bbb Z}/d{\Bbb Z})^n$.

Now, set   $w\in W_n$ parameterized by $(i_1,\dots, i_n)\in X_n^n$. Then from Lemma \ref{aplicacion} and the definition of the elements $u_k^r$'s, we get
$$
(u_1^0\otimes \dots \otimes u_n^0)\Phi_w=u_{i_1}^0\otimes \dots \otimes u_{i_n}^0.
$$
On the other hand, by using  (\ref{corre}) we have:
$$
(u_{i_1}^0\otimes \cdots \otimes u_{i_n}^0)T_1^{m_1}\cdots T_n^{m_n}
=
u_{i_1}^{m_1}\otimes \cdots \otimes u_{i_n}^{m_n}
$$
where $m:= (m_1, \ldots , m_n)$ runs in $({\Bbb Z}/d{\Bbb Z})^n$.
Thus, evaluating equation (\ref{principal})
 in $u_1^0\otimes \cdots \otimes u_n^0$ we obtain
$$
\sum_{m, i}\lambda_{m,i} u_{i_1}^{m_1}\otimes \dots \otimes u_{i_n}^{m_n}=0
$$
where $i:=(i_1, \ldots i_n)$ runs in $X_n^n$ and $m:=(m_1,\dots, m_n)$ runs in $({\Bbb Z}/d{\Bbb Z})^n$. Therefore, $\lambda_{m, i} = 0$ for all $i$ and $m$ since the left side of the last equation is a linear combination of elements of the basis $\mathcal{D}^{\otimes n}$.
\end{proof}
In particular, the above theorem implies the following corollary.
\begin{corollary}\label{faithful}
The representation $\Phi$ is faithful.
\end{corollary}

\subsection{\it The basis $\mathsf{C}_n$}

For all $1\leq k\leq n$, let us define inductively  the sets $M_{d,k}$ by
$$
M_{d,1} = \{t_1^m, t_1^m b_1 \,;\, 0 \leq m \leq d-1\}
$$
and
$$
M_{d,k}=\{t_k^m, t_k^mb_{k}, g_{k-1} x\,;\, x \in M_{d,k-1},\, 0\leq m\leq d-1\}
\quad \quad \mbox{for all $2\leq k\leq n$.}
$$

\begin{definition}\label{Cn}\rm
We define $\mathsf{C}_n$ as the subset of $\Y$ formed by the following elements:
\begin{equation}\label{mon1}
  \mathfrak{m}_1\mathfrak{m}_2\cdots \mathfrak{m}_n
\end{equation}
where $\mathfrak{m}_i\in M_{d,i}$.
\end{definition}

To prove that  $\mathsf{C}_n$ is a linearly spanning set  we will need some formulas of multiplication among the defining generators of $\Y$ and the elements $M_{d,k}$. These are given in Lemmas \ref{relations1C}--\ref{relations3C} below.
 Now notice that  every element of $M_{d,k}$ has the form $\mathfrak{m}_{k,j,m}^+$ or  $\mathfrak{m}_{k,j,m}^-$ with $j\leq k$ and $0\leq m\leq d-1$,  where
$$
\mathfrak{m}_{k,k,m}^+:=t_k^m, \qquad \mathfrak{m}_{k,j,m}^+ := g_{k-1}\cdots g_jt_j^m\quad \text{for}\ j<k
$$
and
$$
\mathfrak{m}_{k,k,m}^-:=t_k^mb_k, \qquad \mathfrak{m}_{k,j,m}^- := g_{k-1}\cdots g_jb_{j}t_j^m\quad \text{for}\ j<k.
$$

\begin{lemma}\label{relations1C}
The following hold:
 \begin{enumerate}
 \item[(i)]
 $\mathfrak{m}_{n,k,m}^{\pm}t_j=\left\{\begin{array}{lc}
 t_{j-1}\mathfrak{m}_{n,k,m}^{\pm}  & \mbox{for $j>k$}\vspace{3mm}\\
\mathfrak{m}_{n,k,m+1}^{\pm}  & \mbox{for $j=k$} \vspace{3mm}\\
t_j\mathfrak{m}_{n,k,m}^{\pm}  & \mbox{for $j<k$}
\end{array}\right.\vspace{3mm}
$
\item[(ii)]
$
t_n\mathfrak{m}_{n,k,m}^{\pm}=\mathfrak{m}_{n,k,m+1}^{\pm}.
$
 \end{enumerate}
\end{lemma}
\begin{proof}
The proof  is straightforward.
\end{proof}

\begin{lemma}\label{relations2C}
The following hold:
$$
\mathfrak{m}_{n,k,m}^{\pm}g_j=
\left\{\begin{array}{ll}
g_{j-1}\mathfrak{m}_{n,k,m}^{\pm}  & \mbox{for $j>k$}\vspace{3mm}\\
\mathfrak{m}_{n,k+1,m}^{\pm}+d^{-1}(\U-\U^{-1})\sum_s t_{j}^{m-s}\mathfrak{m}_{n,k,s}^{\pm}  & \mbox{for $j=k$}\vspace{3mm} \\
\mathfrak{m}_{n,k-1,m}^{\pm}  & \mbox{for $j=k-1$}\vspace{3mm}\\
g_j\mathfrak{m}_{n,k,m}^{\pm}& \mbox{for $j<k-1$}.
\end{array}\right.
$$
\end{lemma}
\begin{proof}
 The positive case follows directly from Lemma \ref{relations2D} (ii), since $\mathfrak{m}_{n,k,m}^{+}=\mathfrak{n}_{n,k,m}^{+}$. For the negative case we have:
   \begin{eqnarray*}
    \mathfrak{m}_{n,k,m}^{-}g_k &=&g_{n-1}\cdots g_{k}b_kt_k^{m}g_k  \\
     &=& g_{n-1}\cdots g_{1}b_1g_1^{-1}\cdots g_{k-1}^{-1} t_{k}^{m}g_k \\
     &=& g_{n-1}\cdots g_{1}b_1g_1^{-1}\cdots g_{k-1}^{-1}g_k t_{k+1}^{m} \\
     &=& g_{n-1}\cdots g_{1}b_1g_1^{-1}\cdots g_{k-1}^{-1}g_k^{-1} t_{k+1}^{m}+(\U-\U^{-1})g_{n-1}\cdots g_{1}b_1g_1^{-1}\cdots g_{k-1}^{-1}e_k t_{k+1}^{m} \\
     &=& g_{n-1}\cdots g_{1}b_1g_1^{-1}\cdots g_{k-1}^{-1}g_k^{-1} t_{k+1}^{m} +\frac{1}{d}(\U-\U^{-1})\sum_sg_{n-1}\cdots g_{1}b_1g_1^{-1}\cdots g_{k-1}^{-1}t_k^s t_{k+1}^{m-s} \\
        &=& \mathfrak{m}_{n,k+1,m}^{-}+\frac{1}{d}(\U-\U^{-1})\sum_s t_{k}^{m-s} \mathfrak{m}_{n,k,s}^{-}.
  \end{eqnarray*}
\end{proof}

\begin{lemma}\label{relations3C}
The following hold:
\begin{enumerate}
\item[(i)]
$
\mathfrak{m}_{n,k,m}^+b_1=
\left\{\begin{array}{cc}
\mathfrak{m}_{n,k,m}^-  & \mbox{for $k=1$}\vspace{3mm}\\

 b_1\mathfrak{m}_{n,k,m}^+  & \mbox{for $k>1$}
\end{array}\right.
$
\bigskip
\item[(ii)] $\mathfrak{m}_{n,k,m}^-b_1=\left\{\begin{array}{cc}
\mathfrak{m}_{n,k,m}^+ +d^{-1}(\V-\V^{-1})\sum_s\mathfrak{m}_{n,k,s}^+.  & \mbox{for $k=1$}\vspace{3mm}\\

 b_1\mathfrak{m}_{n,k,m}^- + d^{-1}(\U-\U^{-1})\sum_s \mathfrak{p}_{k,s} & \mbox{for $k>1$}
\end{array}\right.$\vspace{3mm}\\
where $\mathfrak{p}_{k,s}[b_1t_1^{m-s}g_1^{-1}\dots g_{k-2}^{-1}\mathfrak{m}_{n,1,s}^--t_1^{m-s}g_1^{-1}\dots
 g_{k-2}^{-1}(\mathfrak{m}_{n,1,s}^-b_1 )]$. In particular we have:
$$
(b_nt_n^{m})b_1=b_1(b_nt_n^{m})+\frac{1}{d}(\U-\U^{-1})\sum_s\mathfrak{p}_{n,s}.
$$
\end{enumerate}
\end{lemma}

\begin{proof}
The claim of (i) is straightforward. To prove claim (ii), we note first that $\mathfrak{m}_{n,k,m}^- b_1= g_{n-1}\dots (g_{1}b_1g_1^{-1}b_1)g_2^{-1}\dots g_{k-1}^{-1} t_{k}^{m} $.  Then, splitting $g_1^{-1}$ according to (\ref{inverse}) and invoking (\ref{braid4}) we deduce:
\begin{eqnarray*}
 \mathfrak{m}_{n,k,m}^- b_1
 & = & g_{n-1}\dots g_2 (b_1g_{1}b_1g_1)g_2^{-1}\dots g_{k-1}^{-1} t_{k}^{m} -(\U-\U^{-1})g_{n-1}\dots g_{1}be_1b_1g_2^{-1}\dots g_{k-1}^{-1} t_{k}^{m}\\
 & = & b_1g_{n-1}\dots g_2 (g_{1}b_1g_1)g_2^{-1}\dots g_{k-1}^{-1} t_{k}^{m} -(\U-\U^{-1})g_{n-1}\dots g_{1}be_1g_2^{-1}\dots g_{k-1}^{-1} t_{k}^{m}b_1.
\end{eqnarray*}
By using again (\ref{inverse}), we write the second $g_1$ that appears  inside the  parenthesis above in terms of $g_1^{-1}$. So, we obtain:
\begin{eqnarray*}
\mathfrak{m}_{n,k,m}^- b_1
&=&
 b_1 g_{n-1}\dots g_{1}b_1g_1^{-1}g_2^{-1}\dots g_{k-1}^{-1} t_{k}^{m}+(\U-\U^{-1})b_1 g_{n-1}\dots g_{1}b_1e_1g_2^{-1}\dots g_{k-1}^{-1} t_{k}^{m} - \\
& &
(\U-\U^{-1})g_{n-1}\dots g_{1}b_1e_1g_2^{-1}\dots g_{k-1}^{-1} t_{k}^{m}b_1.
\end{eqnarray*}
Hence
\begin{eqnarray*}
\mathfrak{m}_{n,k,m}^- b_1
& = &
(\U-\U^{-1})\left[b_1 g_{n-1}\dots g_{1}b_1e_1 t_{2}^{m}g_2^{-1}\dots g_{k-1}^{-1} -g_{n-1}\dots g_{1}be_1 t_{2}^{m}g_2^{-1}\dots g_{k-1}^{-1}b_1\right]\\
 &  &
 + b_1\mathfrak{m}_{n,k,m}^-.
\end{eqnarray*}
Now, by using the definition of $e_i$, we obtain:
\begin{eqnarray*}
 b_1 g_{n-1}\dots g_{1}b_1e_1 t_{2}^{m}g_2^{-1}\dots g_{k-1}^{-1}
 & = &
  \sum_s b_1 g_{n-1}\dots g_{1}b_1t_1^s t_{2}^{m-s}g_2^{-1}\dots g_{k-1}^{-1}\\
  & = &
  \sum_s b_1 t_{1}^{m-s}g_1^{-1}\dots g_{k-2}^{-1}\mathfrak{m}_{n,1,s}^-.
\end{eqnarray*}
In the same way we obtain:
\begin{eqnarray*}
g_{n-1}\dots g_{1}b_1t_1^s t_{2}^{m-s}g_2^{-1}\dots g_{k-1}^{-1}b_1
& = &
\sum_s g_{n-1}\dots g_{1}b_1t_1^s t_{2}^{m-s}g_2^{-1}\dots g_{k-1}^{-1}b_1 \\
& = &
t_{1}^{m-s}g_1^{-1}\dots g_{k-2}^{-1}\mathfrak{m}_{n,1,s}^-b_1
\end{eqnarray*}
Therefore,  the proof follows.
\end{proof}

\begin{proposition}\label{basisC}
The set $\mathsf{C}_n$ is a basis  for $\Y$.
\end{proposition}
\begin{proof}
We can prove that $\mathsf{C}_n$ is a spanning set for $\Y$ analogously to the proof of Proposition \ref{spanD}, but by using now Lemmas \ref{relations1C}--\ref{relations3C} instead of Lemmas \ref{relations1D} and \ref{relations2D}. Now, the cardinal of $\mathsf{C}_n$ is $2^nd^dn!$, hence the proof follows.
\end{proof}
We shall close the subsection with a lemma, which will be used
 in Section 6.

\begin{lemma}\label{tr1}
For $k\geq 2$ and $X\in \Y$ the following identities hold:
\begin{enumerate}
  \item[(i)] $\mathfrak{m}_{k,j,m}^-b_k=b_{k-1}\mathfrak{m}_{k,j,m}^-\quad\mbox{for $j\leq k-1$}$
  \item[(ii)]  $g_{n-1}Xg_{n-1}^{-1}=g_{n-1}^{-1}Xg_{n-1}+(\U-\U^{-1})(e_{n-1}Xg_{n-1}-g_{n-1}Xe_{n-1})$
   \item[(iii)] $g_{k-1}^2b_{k-1}g_{k-1}^{-1}=b_{k-1}g_{k-1}-(\U-\U^{-1}) b_{k-1}e_{k-1}+(\U-\U^{-1})e_{k-1}b_k$
  \end{enumerate}
\end{lemma}
\begin{proof}
To prove claim (i) we use  Lemmas \ref{relations1C}--\ref{relations3C}. More precisely, we have:
\begin{eqnarray*}
  \mathfrak{m}_{k,j,m}^-b_k &=& g_{k-1}g_{k-2}\dots g_1b_1g_1^{-1}\dots g_{j-1}^{-1}t_j^{m}(g_{k-1}\dots g_1b_1g_1^{-1}\dots g_{k-1}^{-1}) \\
  &=& g_{k-1}\dots g_1b_1g_1^{-1}\dots g_{j-1}^{-1}(g_{k-1}\dots g_1b_1g_1^{-1}\dots g_{k-1}^{-1}) t_j^{m}\\
   &=& g_{k-1}\dots g_1b_1(g_{k-1}\dots g_2g_1b_1g_1^{-1}\dots g_{k-1}^{-1})g_1^{-1}\dots g_{j-1}^{-1}t_j^{m}\\
   &=& g_{k-2}\dots g_1 g_{k-1}\dots g_2g_1b_1 g_1b_1g_1^{-1}\dots g_{k-1}^{-1})g_1^{-1}\dots g_{j-1}^{-1} t_j^{m}\\
   &=& g_{k-2}\dots g_1 g_{k-1}\dots g_2 b_1g_1b_1g_1g_1^{-1}g_2^{-1}\dots g_{k-1}^{-1})g_1^{-1}\dots g_{j-1}^{-1} t_j^{m} \\
   &=&g_{k-2}\dots g_1 b_1g_{k-1}\dots g_1b_1g_2^{-1}\dots g_{k-1}^{-1})g_1^{-1}\dots g_{j-1}^{-1}t_j^{m}   \\
   &=& (g_{k-2}\dots g_1 b_1g_1^{-1}\dots g_{k-2}^{-1})(g_{k-1}g_{k-2}\dots g_1b_1g_1^{-1}\dots g_{j-1}^{-1})  t_j^{m} \\
   &=& b_{k-1}\mathfrak{m}_{k,j,m}^-.
\end{eqnarray*}

For claim (ii) we have by (\ref{inverse}):
\begin{eqnarray*}
  g_{n-1}Xg_{n-1}^{-1} &=& (g_{n-1}^{-1} + (\U -\U^{-1})e_{n-1}) X(g_{n-1}- (\U -\U^{-1})e_{n-1}) \\
   &=&g_{n-1}^{-1}Xg_{n-1}+ (\U -\U^{-1})e_{n-1}Xg_{n-1} -(\U -\U^{-1})g_{n-1}^{-1}Xe_{n-1}\\
   & & -(\U -\U^{-1})^2e_{n-1}Xe_{n-1}.
   \end{eqnarray*}
Then, by expanding the $g_{n-1}^{-1}$ above, the result follows. Finally, we can prove claim (iii) similarly, but by using now (\ref{quadraticU}) and (\ref{inverse}).
\end{proof}

\section{A Markov trace on  $\Y$}

The section is devoted to proving that the tower of algebras (\ref{tower}) associated to the algebras $\Y$ supports a Markov trace (Theorem \ref{trace}). This fact is proved by using the method of relative traces, cf. \cite{aijuMMJ, chpoIMRN}. Probably this method is due to  A. P. Isaev and O. V. Ogievetsky, see for example \cite{isog}. In few words, the method consists in  constructing a certain family of linear maps
${\rm tr}_n: \Y \longrightarrow {\rm Y}_{d,n-1}^{\mathtt{B}}$, called {\it relative traces}, which builds step by step the  desired Markov properties. The  Markov trace on $\Y$ is defined by  ${\rm Tr}_n := {\rm tr}_1\circ \cdots \circ {\rm tr}_n$.

\subsection{}

Let $z$ be an indeterminate and denote by ${\Bbb L}$  the field of rational functions ${\Bbb K}(z) = {\Bbb C}(\U,\V, z)$. We work now on the algebra
${\Bbb L}\otimes_{\Bbb K} \Y$ which, for simplicity, we denote again  by
 $\Y$. Notice that ${\Bbb L}\otimes_{\Bbb K} {\Bbb K}={\Bbb L}$. Consequently,
  ${\rm Y}_{d,0}^{\mathtt{B}}$ is taken as $\Bbb{L}$.

   We set
$x_0:=1$ and  from now on we fix non--zero parameters  $x_1,\dots, x_{d-1}, y_0, \dots ,y_{d-1}$   in $\mathbb{L}$.

\smallbreak
\begin{definition}\label{transitionmaps} \rm
For $n\geq 1$, we define  the linear functions ${\rm tr}_n: \Y \longrightarrow {\rm Y}_{d,n-1}^{\mathtt{B}}$  as follows. For $n=1$,  ${\rm tr}_1 (t_1^{a_1}) = x_{a_1}$ and ${\rm tr}_1 (b_1t_1^{a_1}) = y_{a_1}$. For $n\geq 2$, we define   ${\rm tr}_n$ on the basis $\mathsf{C}_n$ of $\Y$ by:

\begin{equation}\label{tracedef}
 {\rm tr}_n(w \mathfrak{m}_n)
=\left\{\begin{array}{ll}
x_{m}w & \quad \mbox{for}\quad \mathfrak{m}_n=t_{n}^{m} \\
y_{m}w & \quad \mbox{for}\quad \mathfrak{m}_n=b_{n}t_{n}^{m} \\
 z w \mathfrak{m}_{n-1,k,m}^{\pm} & \quad  \mbox{for} \quad \mathfrak{m}_n=\mathfrak{m}_{n,k, m}^{\pm}
\end{array}\right.
\end{equation}
where $w:= \mathfrak{m}_1\cdots \mathfrak{m}_{n-1}\in {\mathsf C}_{n-1}$.   Note that (\ref{tracedef}) also holds for $w\in \Z$, since ${\mathsf C}_{n-1}$ is a basis for $\Z$.
\end{definition}

\begin{lemma}\label{traza1}
For all $X,Z\in \Z$ and $Y\in \Y$, we have:
\begin{itemize}
\item[(i)] ${\rm tr}_n(YZ)={\rm tr}_n(Y)Z$
\item[(ii)] ${\rm tr}_n(XY)=X{\rm tr}_n(Y)$
\item[(iii)] ${\rm tr}_n(XYZ)=X{\rm tr}_n(Y)Z$.
\end{itemize}
\end{lemma}

\begin{proof}
 For proving claim (i) notice that, due to the linearity  of ${\rm tr}_n$, we can suppose that $ Z$ is a  defining generator of ${\rm Y}_{d,n-1}^{\mathtt{B}}$ and $Y=w \mathfrak{m}_n\in \mathsf{C}_n$, with $w\in \mathsf{C}_{n-1}$. Further, to prove the claim  we shall distinguish the $Y$'s according to the possibilities of $\mathfrak{m}_n$.

\noindent $\bullet$ For $\mathfrak{m}_n=t_n^{m}$, we have
$YZ= wt_n^{m}Z= wZt_n^{m} $, then ${\rm tr}_n(wZt_n^{m}) =x_{m}wZ$ since $wZ\in \Z$. Hence, ${\rm tr}_n(YZ) ={\rm tr}_n(Y)Z $.

\noindent $\bullet$ For $\mathfrak{m}_n=b_nt_n^{m}$, we consider first $Z\in\{t_1,\ldots , t_{n-1}, g_1,\ldots,g_{n-2}\} $. Then using Lemma \ref{relations1C} and \ref{relations2C} we have: $YZ= wb_nt_n^{m}Z =wZb_nt_n^{m}$. Hence,
$$
{\rm tr}_n(YZ) ={\rm tr}_n (wZb_nt_n^{m}) = y_{m}wZ = {\rm tr}_n(Y) Z.
$$
Suppose now $Z=b_1$. By definition, ${\rm tr}_n(wb_nt_n^{m}b_1)=w{\rm tr}_n(b_nt_n^{m}b_1)$. Then by Lemma \ref{relations3C}
$$
  {\rm tr}_n(wb_nt_n^{m}b_1) =  w{\rm tr}_n\left(b_1b_nt_n^{m}+\frac{1}{d}(\U-\U^{-1}) A\right)
$$
where
$$
A:=\sum_s\left(b_1t_1^{m-s}g_1^{-1}\dots g_{n-2}^{-1}\mathfrak{m}_{n,1,s}^--t_1^{m-s}g_1^{-1}\dots g_{n-2}^{-1}(\mathfrak{m}_{n,1,s}^-b_1 )\right)
$$
But, we have
\begin{eqnarray*}
 {\rm tr}_n (A) & = & z \sum_s\left(b_1t_1^{m-s}g_1^{-1}\dots g_{n-2}^{-1}\mathfrak{m}_{n-1,1,s}^--t_1^{m-s}g_1^{-1}\dots g_{n-2}^{-1}(\mathfrak{m}_{n-1,1,s}^-b_1 )\right) \\
   &=& z\sum_s\left(b_1t_1^{m-s}b_1t_1^s-t_1^{m-s}b_1t_1^sb_1\right)=0.
\end{eqnarray*}
Therefore,
$$
{\rm tr}_n(YZ)={\rm tr}_n(wb_nt_n^{m}b_1)= {\rm tr}_n(w b_1b_nt_n^{m}) =y_m wb_1 = {\rm tr}_n(Y)Z.
$$

\noindent $\bullet$ For $\mathfrak{m}_n=\mathfrak{m}_{n,k,m}^{\pm}$, with $k<n$, we have:

\smallbreak
\noindent $\ast$ If $Z=t_j$ with $j\in \{1,\dots, {n-1}\}$, the claim follows directly from  (i) Lemma \ref{relations1C}. For example, for $j>k$, we have
\begin{eqnarray*}
  {\rm tr}_n(w\mathfrak{m}_{n,k,m}^{\pm}t_j) &=&  {\rm tr}_n(wt_{j-1}\mathfrak{m}_{n,k,m}^{\pm}) \\
   &=& zwt_{j-1}\mathfrak{m}_{n-1,k,m}^{\pm}\\
   &=&  zw\mathfrak{m}_{n-1,k,m}^{\pm}t_{j}={\rm tr}_n(Y)Z.
\end{eqnarray*}
We can proceed in similar way  for the other cases for $j$.

\smallbreak
\noindent $\ast$ If $Z=g_j$ with $j\in \{1,\dots, n-2\}$. The claim follows by using the formulas of Lemma \ref{relations2C}. Below, we show only the prove of the case  $j=k$, since that the other cases for $j$ follows easily. We have:
 \begin{eqnarray*}
  {\rm tr}_n(w\mathfrak{m}_{n,k,m}^{\pm}g_j) &=&  {\rm tr}_n\left(w\left[ \mathfrak{m}_{n,k+1,m}^{\pm}+\frac{1}{d}(\U-\U^{-1})\sum_s t_{j}^{m-s}\mathfrak{m}_{n,k,s}^{\pm})\right]\right) \\
   &=&\left(zw\left[ \mathfrak{m}_{n-1,k+1,m}^{\pm}+\frac{1}{d}(\U-\U^{-1})\sum_s t_{j}^{m-s}\mathfrak{m}_{n-1,k,s}^{\pm})\right]\right)\\
   &=&  zw\mathfrak{m}_{n-1,k,m}^{\pm}g_{j}={\rm tr}_n(Y)Z.
\end{eqnarray*}

\smallbreak
\noindent $\ast$ If $Z=b_1$, we deduce the claim directly from  (ii) Lemma \ref{relations1C}.

\smallbreak

To prove (ii), by using the linearity of ${\rm tr}_n$, we can suppose  again that  $X$ stands for the defining generators of ${\rm Y}_{d,n-1}^{\mathtt{B}}$ and $Y=w \mathfrak{m}_n\in \mathsf{C}_n$, with $w\in \mathsf{C}_{n-1}$. Note now that $Xw\in \Z$. Hence  claim (ii) follows directly from the definition of ${\rm tr}_n$.

\smallbreak

Finally, claim (iii) is  a combination of  claims (i) and (ii).
\end{proof}

\begin{lemma}\label{traza 2}
  For every $n\geq 1$ and $X\in \Y$, we have that
  $${\rm tr}_n(Xt_n)={\rm tr}_n(t_n X)$$
\end{lemma}

\begin{proof}
  As we know, from linearity of the trace is enough consider $X$ in $\mathsf{C}_n$, then we have $X=w\mathfrak{m}_n$, with $w\in \mathsf{C}_{n-1}$. Whenever  $\mathfrak{m}_n=b_nt_n^{m}$ or $t_n^{m}$ the result is clear, since $t_n$ commute with $X$. So,  suppose  $\mathfrak{m}_n=\mathfrak{m}_{n,k,m}^{\pm}$. Then, from  Lemma \ref{relations1C}, we obtain
$$
{\rm tr}_n(Xt_n) = {\rm tr}_n(w\mathfrak{m}_{n,k,m}^{\pm}t_n)
= zw t_{n-1}\mathfrak{m}_{n-1,k,m}^{\pm}
=  zw \mathfrak{m}_{n-1,k,m+1}^{\pm}.
$$
On the other hand, we have
$$
      {\rm tr}_n(t_n X) = {\rm tr}_n(wt_n\mathfrak{m}_{n,k,m}^{\pm})
       = {\rm tr}_n(w\mathfrak{m}_{n,k,m+1}^{\pm})
       =  zw \mathfrak{m}_{n-1,k,m+1}^{\pm}.
$$
Thus, the proof of the lemma follows.

\end{proof}
\begin{lemma}\label{exg}
  For $n\geq 2$, $X\in \Z$ and $Y\in \Y$, we have:
  \begin{itemize}
     \item[(i)]${\rm tr}_n(e_{n-1}Xg_{n-1}) = {\rm tr}_n(g_{n-1}Xe_{n-1})$
    \item[(ii)]${\rm tr}_{n-1} {\rm tr}_n(e_{n-1}Y)= {\rm tr}_{n-1}{\rm tr}_n(Y e_{n-1}).$
  \end{itemize}
\end{lemma}

\begin{proof}
We prove (i). Expanding the left side and using Lemma \ref{traza1}, we have:
$$
  {\rm tr}_n(e_{n-1}Xg_{n-1}) =  \frac{1}{d}\sum_s {\rm tr}_n(t_{n-1}^s t_{n}^{-s}Xg_{n-1}) =   \frac{1}{d}\sum_s {\rm tr}_n(t_{n-1}^s Xg_{n-1}t_{n-1}^{-s})
     = \frac{1}{d}\sum_s zt_{n-1}^s Xt_{n-1}^{-s}.
$$
Similarly, we expand  the right side obtaining:
$$
{\rm tr}_n(g_{n-1}Xe_{n-1}) = \frac{1}{d}\sum_s zt_{n-1}^{-s} Xt_{n-1}^{s}.
$$
Hence, claim (i) is true.

\smallbreak

To prove claim (ii) we use Lemmas \ref{traza1} and \ref{traza 2}.  Indeed, we have:
\begin{eqnarray*}
 {\rm tr}_{n-1} ({\rm tr}_n(e_{n-1}Y)) &=&  \frac{1}{d}\sum_s{\rm tr}_{n-1}({\rm tr}_n(t_{n-1}^s t_n^{-s}Y))\\
   &=& \frac{1}{d}\sum_s{\rm tr}_{n-1}(t_{n-1}^s {\rm tr}_n(t_n^{-s}Y)) \\
   &=& \frac{1}{d}\sum_s{\rm tr}_{n-1}(t_{n-1}^s {\rm tr}_n(Y t_n^{-s}))  \\
   &=&  \frac{1}{d}\sum_s{\rm tr}_{n-1}({\rm tr}_n(Y t_n^{-s})t_{n-1}^s)  \\
   &=& \frac{1}{d}\sum_s{\rm tr}_{n-1}({\rm tr}_n(Y t_n^{-s}t_{n-1}^s)) \\
   &=& {\rm tr}_{n-1} ({\rm tr}_n(Y e_{n-1})).
\end{eqnarray*}
\end{proof}

\begin{lemma}\label{traza3}
  For $n\geq 2$ and $X\in \Z$. We have
  $$
  {\rm tr}_n(g_{n-1}Xg_{n-1}^{-1}) = {\rm tr}_{n-1}(X)={\rm tr}_n(g_{n-1}^{-1}Xg_{n-1}).
  $$
\end{lemma}

\begin{proof}
As before, we can suppose $X= w \mathfrak{m}_{n-1}$ with $w\in \mathsf{C}_{n-2}$. We will check the first equality by distinguishing the possibilities for $\mathfrak{m}_{n-1}$. For $\mathfrak{m}_{n-1}=t_{n-1}^{m}$ the claim is  only a direct computation.  For $\mathfrak{m}_{n-1}=b_{n-1}t_{n-1}^{m}$, we have
  $$
  {\rm tr}_n(g_{n-1}wb_{n-1}t_{n-1}^{m}g_{n-1}^{-1})={\rm tr}_n(wb_{n}t_{n}^{m})=y_{m}w={\rm tr}_{n-1}(X).
  $$
Finally, for $\mathfrak{m}_{n-1}=\mathfrak{m}_{n-1,k,m}^{\pm}$, we have:
  \begin{eqnarray*}
     {\rm tr}_n(g_{n-1}w\mathfrak{m}_{n-1,k,m}^{\pm}g_{n-1}^{-1}) &=& {\rm tr}_n(w\mathfrak{m}_{n,k,m}^{\pm}g_{n-1}^{-1}) \\
     &=& {\rm tr}_n(w g_{n-2}^{-1}\mathfrak{m}_{n,k,m}^{\pm})  \\
     &=& zw g_{n-2}^{-1}\mathfrak{m}_{n-1,k,m}^{\pm}  \\
     &=& zw \mathfrak{m}_{n-2,k,m}^{\pm} \\
     &=&{\rm tr}_{n-1}(X).
  \end{eqnarray*}
Thus,  the proof of the first equality is done.
\smallbreak
From a combination  of   Lemma \ref{tr1} (ii) and  Lemma \ref{exg} (i), we deduce
\begin{equation*}
{\rm tr}_n(g_{n-1}Xg_{n-1}^{-1}) = {\rm tr}_n(g_{n-1}^{-1}Xg_{n-1}).
\end{equation*}
 Hence, from (i) the second equality follows.
\end{proof}

\begin{lemma}\label{trazaimp}
For all $X\in \Y$, we have
$$
{\rm tr}_{n-1}({\rm tr}_n(Xg_{n-1}))={\rm tr}_{n-1}({\rm tr}_n(g_{n-1}X)).
$$
\end{lemma}
\begin{proof}
Again, from the linearity of ${\rm tr}_n$ it is enough to consider $X$  in the basis $\mathsf{C}_n$. Set  $X=w \mathfrak{m}_n$, with $w\in \mathsf{C}_{n-1}$. We are going to prove  the statement by distinguishing  according to the possibilities for $\mathfrak{m}_n$.

\smallbreak

\noindent $\bullet$ For $\mathfrak{m}_n=t_n^{m}$, the claim  follows from Lemma \ref{traza 2}.

\smallbreak

\noindent $\bullet$ For $\mathfrak{m}_n=\mathfrak{m}_{n,k,m}^{\pm}$, we note that, using the formula (\ref{inverse}) for the inverse of $g_{n-1}$ and  Lemma \ref{exg} (ii), we obtain the following:
$$
{\rm tr}_{n-1}({\rm tr}_n(Xg_{n-1}^{-1}))={\rm tr}_{n-1}({\rm tr}_n(Xg_{n-1})).
$$
Now, for the left side of this  equality, we have:
\begin{eqnarray*}
  {\rm tr}_{n-1}({\rm tr}_n(Xg_{n-1}^{-1})) &=&  {\rm tr}_{n-1}({\rm tr}_n(w\mathfrak{m}_{n,k,m}^{\pm}g_{n-1}^{-1}))  \\
   &=&  {\rm tr}_{n-1}({\rm tr}_n(wg_{n-1}\mathfrak{m}_{n-1,k,m}^{\pm}g_{n-1}^{-1}))  \\
   &=&   {\rm tr}_{n-1}(w{\rm tr}_n(g_{n-1}\mathfrak{m}_{n-1,k,m}^{\pm}g_{n-1}^{-1})) \\
   &=&  {\rm tr}_{n-1}(w{\rm tr}_{n-1}(\mathfrak{m}_{n-1,k,m}^{\pm})) \\
   &=& {\rm tr}_{n-1}(w){\rm tr}_{n-1}(\mathfrak{m}_{n-1,k,m}^{\pm}).
\end{eqnarray*}
In the same manner one obtains this last expression for ${\rm tr}_{n-1}({\rm tr}_n(Xg_{n-1}^{-1}))$. In consequence the  claim holds.

\noindent $\bullet$ Finally, for $\mathfrak{m}_{n}=b_{n}t_n^{m}$ we separate the proof depending on the form of $w$ in $X=w\mathfrak{m}_{n}$.

\smallbreak
\noindent $\ast$ Suppose that $w=w^{\prime}t_{n-1}^{\beta}$, with $w^{\prime}\in \mathsf{C}_{n-2}$. Then,   we have:
\begin{eqnarray*}
  {\rm tr}_{n-1}({\rm tr}_n(Xg_{n-1}))&=&  {\rm tr}_{n-1}({\rm tr}_n(w^{\prime}t_{n-1}^{\beta} b_{n}t_n^{m} g_{n-1}))\\
     &=& {\rm tr}_{n-1}({\rm tr}_n(w^{\prime}t_{n-1}^{\beta} b_{n} g_{n-1}t_{n-1}^{m}))  \\
   &=&{\rm tr}_{n-1}({\rm tr}_n(w^{\prime}t_{n-1}^{\beta}g_{n-1} b_{n-1}t_{n-1}^{m}))   \\
   &=& z {\rm tr}_{n-1}(w^{\prime}t_{n-1}^{\beta} b_{n-1}t_{n-1}^{m}) \\
   &=& zx_{m+\beta}w^{\prime}.
   \end{eqnarray*}
On other hand:
   \begin{eqnarray*}
  {\rm tr}_{n-1}({\rm tr}_n(g_{n-1}X))
  &=&
  {\rm tr}_{n-1}({\rm tr}_n(g_{n-1}w^{\prime}t_{n-1}^{\beta} b_{n}t_n^{m} )) \\
  &=&
  {\rm tr}_{n-1}({\rm tr}_n(w^{\prime}g_{n-1} b_{n}t_{n-1}^{\beta}t_n^{m} ))  \\
  &=&
  {\rm tr}_{n-1}({\rm tr}_n(w^{\prime}g_{n-1}^2 b_{n-1}g_{n-1}^{-1}t_{n-1}^{\beta}t_n^{m} ))
  \end{eqnarray*}
 Then using Lemma \ref{tr1} (iii), we obtain
$$
 {\rm tr}_{n-1}({\rm tr}_n(g_{n-1}X)) = A - (\U-\U^{-1})B + (\U-\U^{-1}) C
$$
where
$$
   \begin{array}{ccl}
   A & := & {\rm tr}_{n-1}({\rm tr}_n(w^{\prime} b_{n-1}g_{n-1}t_{n-1}^{\beta}t_n^{m}))\\
   B & := & {\rm tr}_{n-1}({\rm tr}_n(w^{\prime}b_{n-1}e_{n-1}t_{n-1}^{\beta}t_n^{m}))\\
   C & := & {\rm tr}_{n-1}({\rm tr}_n(w^{\prime}e_{n-1}b_nt_{n-1}^{\beta}t_n^{m})).
   \end{array}
   $$
We will compute the values of $A$, $B$ and $C$. A direct computation shows that:
$$ A= {\rm tr}_{n-1}({\rm tr}_n(w^{\prime} b_{n-1}t_{n-1}^{m}g_{n-1}t_{n-1}^{\beta}))=
z{\rm tr}_{n-1}({\rm tr}_n(w^{\prime} b_{n-1}t_{n-1}^{m +\beta})) = zx_{m+\beta}w^{\prime}.
$$
Expanding $e_{n-1}$ in $B$, we get:
\begin{eqnarray*}
 B
&=&
\frac{1}{d}(\U-\U^{-1})\sum_{s}{\rm tr}_{n-1}({\rm tr}_n(w^{\prime} b_{n-1}t_{n-1}^st_{n}^{-s}t_{n-1}^{\beta}t_n^{m})) \\
      &=&\frac{1}{d}\sum_{s}x_{m -s}{\rm tr}_{n-1}(w^{\prime}b_{n-1}t_{n-1}^{\beta+s} )
 \end{eqnarray*}
 Then
 $$
 B =\frac{1}{d}\sum_{s}x_{m -s}y_{\beta+s}w^{\prime}
$$
By expanding also $e_{n-1}$ in $C$, we have:
$$
C = \frac{1}{d}\sum_s{\rm tr}_{n-1}({\rm tr}_n(w^{\prime}t_{n-1}^st_n^{-s}b_{n}t_{n-1}^{\beta}t_n^{m}))
= \frac{1}{d}\sum_s y_{m-s}x_{\beta+s}w^{\prime}
= \frac{1}{d}\sum_{r}y_{\beta+r}x_{m-r}w^{\prime}.
$$
(notice that the last equality is obtain by making $s=-r+ m - \beta$). Thus $B=C$, this imply
$
 {\rm tr}_{n-1}({\rm tr}_n(Xg_{n-1})) = A = {\rm tr}_{n-1}({\rm tr}_n(g_{n-1}X)).
$

\smallbreak
\noindent $\ast$ Suppose $\mathfrak{m}_{n-1}=w^{\prime}b_{n-1}t_{n-1}^{\beta}$, with $w^{\prime}\in \mathsf{C}_{n-2}$. We have:
$$
Xg_{n-1} = w^{\prime}b_{n-1}t_{n-1}^{\beta}b_{n}t_n^{m}g_{n-1}
  = w^{\prime} t_{n-1}^{\beta}b_{n-1}b_{n}g_{n-1}t_{n-1}^{m}
 = w^{\prime}t_{n-1}^{\beta}b_{n-1}g_{n-1}b_{n-1}t_{n-1}^{m}
$$
Then
$$
{\rm tr}_{n-1}({\rm tr}_n(Xg_{n-1})) = z{\rm tr}_{n-1}(w' b_{n-1}^2t_{n-1}^{m+\beta})
$$
On the  other hand:
\begin{eqnarray*}
g_{n-1}X
& = &  g_{n-1}w^{\prime} b_{n-1}t_{n-1}^{\beta}b_{n}t_n^{m}
 = w^{\prime} g_{n-1} b_{n-1}b_{n}t_{n-1}^{\beta}t_n^{m}\\
 & = & w^{\prime} t_{n-1}^{m} g_{n-1} b_{n-1}b_{n}t_{n-1}^{\beta}
= w^{\prime} t_{n-1}^{m} b_{n-1}g_{n-1}b_{n-1}t_{n-1}^{\beta}
\end{eqnarray*}
(in the last equality we have  used (iv) Proposition \ref{braidHBHB}). Hence
$$
 {\rm tr}_{n-1}({\rm tr}_n(g_{n-1}X))
= z {\rm tr}_{n-1}(w'  b_{n-1}^2t_{n-1}^{m + \beta}) =  {\rm tr}_{n-1}({\rm tr}_n(Xg_{n-1})).
$$

\smallbreak

\noindent $\ast$ Suppose  $\mathfrak{m}_{n-1}=w^{\prime}\mathfrak{m}_{n-1,j,\beta}^-$.  We have:
$$
X g_{n-1} =  w^{\prime}\mathfrak{m}_{n-1,j, \beta}^- b_{n}t_n^{m}g_{n-1}
   = w^{\prime}\mathfrak{m}_{n-1,j,\beta}^- b_{n}g_{n-1}t_{n-1}^{m}
   = w^{\prime}\mathfrak{m}_{n-1,j,\beta}^- g_{n-1}b_{n-1}t_{n-1}^{m}.
$$
Then
\begin{eqnarray*}
 {\rm tr}_{n-1}({\rm tr}_n(X g_{n-1}))
   &=&z {\rm tr}_{n-1}(w^{\prime}\mathfrak{m}_{n-1,j,\beta}^- b_{n-1}t_{n-1}^{m}) \\
   &=& z  {\rm tr}_{n-1}(w^{\prime} b_{n-2}t_{n-2}^{m}\mathfrak{m}_{n-1,j,\beta}^-) \\
   &=& z^2  w^{\prime} b_{n-2}t_{n-2}^{m}\mathfrak{m}_{n-2,j,\beta}^-
\end{eqnarray*}
On the other hand, we note that
$$
g_{n-1}X = g_{n-1}w^{\prime}\mathfrak{m}_{n-1,j,\beta}^- b_{n}t_n^{m}
  = w^{\prime} g_{n-1}\mathfrak{m}_{n-1,j, \beta}^- b_{n}t_n^{m}
  = w^{\prime} \mathfrak{m}_{n,j,\beta}^- b_{n}t_n^{m}
  = w^{\prime} b_{n-1}t_{n-1}^{m}\mathfrak{m}_{n,j,\beta}^-.
$$
Then
\begin{eqnarray*}
  {\rm tr}_{n-1}({\rm tr}_n( g_{n-1}X))
   &=& z {\rm tr}_{n-1}(w^{\prime} b_{n-1}t_{n-1}^{m}\mathfrak{m}_{n-1,j,\beta}^-)  \\
   &=& z {\rm tr}_{n-1}(w^{\prime} b_{n-1}g_{n-2}t_{n-2}^{m}\mathfrak{m}_{n-2,j,\beta}^-) \\
   &=& z {\rm tr}_{n-1}(w^{\prime} g_{n-2}b_{n-2}t_{n-2}^{m}\mathfrak{m}_{n-2,j,\beta}^-)  \\
   &=& z^2 w^{\prime} b_{n-2}t_{n-2}^{m}\mathfrak{m}_{n-2,j,\beta}^-.
\end{eqnarray*}
Hence, $ {\rm tr}_{n-1}({\rm tr}_n(X g_{n-1})) ={\rm tr}_{n-1}({\rm tr}_n( g_{n-1}X))$.

\smallbreak

\noindent $\ast$ Finally, let us suppose that  $\mathfrak{m}_{n-1}=\mathfrak{m}_{n-1,j,\beta}^+$. We have
$$
X g_{n-1} =  w^{\prime}\mathfrak{m}_{n-1,j,\beta}^+b_nt_n^{m} g_{n-1}
 = w^{\prime}\mathfrak{m}_{n-1,j,\beta}^+g_{n-1}b_{n-1}t_{n-1}^{m}
$$
Then
\begin{eqnarray*}
   {\rm tr}_{n-1}({\rm tr}_n(X g_{n-1}))
   &=& z{\rm tr}_{n-1}(w'\mathfrak{m}_{n-1,j,\beta}^+b_{n-1}t_{n-1}^{m}) \\
   &=& z {\rm tr}_{n-1}(w'g_{n-2}\mathfrak{m}_{n-2,j,\beta}^+b_{n-1}t_{n-1}^{m})  \\
   &=& z w^{\prime}{\rm tr}_{n-1}(g_{n-2}b_{n-1}t_{n-1}^{m})\mathfrak{m}_{n-2,j,\beta}^+ \\
   &=& z w^{\prime}{\rm tr}_{n-1}(g_{n-2}^2b_{n-2}g_{n-2}^{-1}t_{n-1}^{m})\mathfrak{m}_{n-2,j,\beta}^+
    \end{eqnarray*}
We shall compute now ${\rm tr}_{n-1}(g_{n-2}^2b_{n-2}g_{n-2}^{-1}t_{n-1}^{m})$. To do that, we note that  splitting  the square and recalling the definition of $b_{n-1}$, we can write
 $$
 {\rm tr}_{n-1}(g_{n-2}^2b_{n-2}g_{n-2}^{-1}t_{n-1}^{m})= A-(\U-\U^{-1})B+(\U-\U^{-1})C
 $$
  where
  $$
\begin{array}{ccl}
A & := & {\rm tr}_{n-1}(b_{n-2}g_{n-2}t_{n-1}^{m})=zb_{n-2}t_{n-2}^{m}\\
 B & := &  {\rm tr}_{n-1}(b_{n-2}e_{n-2}t_{n-1}^{m})=\sum_s x_{m-s}b_{n-2}t_{n-2}^s\\
 C & := & {\rm tr}_{n-1}(e_{n-2}b_{n-1}t_{n-1}^{m})=\sum_s y_{m-s}t_{n-2}^s.
\end{array}
$$
Hence
\begin{equation}\label{caseproof1}
{\rm tr}_{n-1}({\rm tr}_n( g_{n-1}X))=z^2w^{\prime}b_{n-2}t_{n-2}^{m}\mathfrak{m}_{n-2,j,\beta}^++z(\U-\U^{-1})w^{\prime}\left( C-B\right)\mathfrak{m}_{n-2,j,\beta}^+
\end{equation}

On the other side, we have
$$
g_{n-1}X =  g_{n-1}w^{\prime}\mathfrak{m}_{n-1,j,\beta}^+b_nt_n^{m}
  =  w^{\prime}g_{n-1}\mathfrak{m}_{n-1,j,\beta}^+b_nt_n^{m}
  =  w^{\prime}g_{n-1}b_n\mathfrak{m}_{n-1,j,\beta}^+t_n^{m}
$$
Then
\begin{eqnarray*}
   {\rm tr}_{n-1}({\rm tr}_n(g_{n-1}X))
   &=& w^{\prime}{\rm tr}_{n-1}({\rm tr}_n(g_{n-1}b_n t_n^{m})\mathfrak{m}_{n-1,j,\beta}^+) \\
   &=&  w^{\prime}{\rm tr}_{n-1}({\rm tr}_n(g_{n-1}^2b_{n-1}g_{n-1}^{-1} t_n^{m})\mathfrak{m}_{n-1,j,\beta}^+)
    \end{eqnarray*}
As before, we split the square and then we deduce the following
\begin{equation}\label{caseproof2}
{\rm tr}_{n-1}({\rm tr}_n(g_{n-1}X)) = A_1 + (\U-\U^{-1})(C_1 -B_1)
\end{equation}
where
$$
\begin{array}{ccl}
A_1 & := & {\rm tr}_{n-1}({\rm tr}_n(b_{n-1}g_{n-1} t_n^{m})\mathfrak{m}_{n-1,j,\beta}^+)
  =  z^2w'b_{n-2}t_{n-2}^{m}\mathfrak{m}_{n-2,j,\beta}^+\\
 B_1 & := & {\rm tr}_{n-1}({\rm tr}_n(b_{n-1}e_{n-1}t_{n}^{m})\mathfrak{m}_{n-1,j,\beta}^+)=z\left[\sum_s x_{m-s}b_{n-2}t_{n-2}^s\right]\mathfrak{m}_{n-2,j,\beta}^+ \\
C_1 & := & {\rm tr}_{n-1}({\rm tr}_n(e_{n-1}b_n t_n^{m})\mathfrak{m}_{n-1,j,\beta}^+)=z\left[\sum_s y_{m-s}t_{n-2}^s\right]\mathfrak{m}_{n-2,j,\beta}^+.
 \end{array}
 $$
Hence, comparing
(\ref{caseproof1}) and (\ref{caseproof2}), the claim follows. Therefore the lemma is proved.
\end{proof}

\subsection{}

 In this subsection we prove that the family  $\{\Y \}_{n\geq 1}$ supports a Markov trace.
Let  ${\rm Tr}_n$ be  the linear map, from $\Y$ to $\Bbb{L}$, defined inductively by setting: ${\rm Tr}_1 = {\rm tr}_1$  and
$$
{\rm Tr}_n = {\rm Tr}_{n-1}\circ {\rm tr}_n \quad \text{for} \quad n\geq 2.
$$
The definition of ${\rm Tr}_n$ says that ${\rm Tr}_n(1)=1$ and that
\begin{equation}\label{Trnk}
{\rm Tr}_n (x) = {\rm Tr}_k (x)\quad \text{for} \quad x\in {\rm Y}_{d, k}^B \quad \text{and} \quad n\geq k.
\end{equation}

Let us denote ${\rm Tr}$ the family $\{{\rm Tr}_n\}_{n\geq 1}$. The following theorem is one of our main results.
\begin{theorem}\label{trace}
  ${\rm Tr}$ is a Markov trace on $\{\Y \}_{n\geq 1}$. That is, for every $n\geq 1$  the linear map ${\rm Tr}_n: \Y \longrightarrow \Bbb{L} $ satisfies  the following rules:
   \begin{enumerate}
    \item[(i)] ${\rm Tr}_n(1)=1$
    \item[(ii)] ${\rm Tr}_{n+1}(Xg_n)=z{\rm Tr}_n(X)$
    \item[(iii)] ${\rm Tr}_{n+1}(Xb_{n+1}t_{n+1}^{m})=y_{m}{\rm Tr}_n(X)$
    \item[(iv)] ${\rm Tr}_{n+1}(Xt_{n+1}^{m})=x_{m}{\rm Tr}_n(X)$
    \item[(v)] ${\rm Tr}_n(XY)={\rm Tr}_n(YX)$
 \end{enumerate}
  where $X,Y\in \Y$.
\end{theorem}

\begin{proof}
Rules (ii)--(iv) are direct  consequences of  Lemma \ref{traza1} (ii). Indeed, for example  for (ii), we have:
$$
{\rm Tr}_{n+1}(Xg_n) = {\rm Tr}_n ( {\rm tr}_{n+1}(Xg_n)) =  {\rm Tr}_n ( X{\rm tr}_{n+1}(g_n))
= {\rm Tr}_n ( Xz)=z{\rm Tr}_n ( X).
$$
  We prove rule (v) by induction on $n$. For $n=1$, the rule  holds since ${\rm Y}_{d,1}^{\mathtt B}$ is commutative. Suppose now that (v) is true for all $k$ less than $n$. We prove it first  for  $Y\in {\rm Y}_{d, n-1}^B$ and $X\in \Y$. We have
  \begin{eqnarray*}
      {\rm Tr}_n(XY) &=& {\rm Tr}_{n-1}({\rm tr}_n(XY)) \\
     &=&{\rm Tr}_{n-1}({\rm tr}_n(X)Y)\quad \mbox{ (by (i) Lemma \ref{traza1})}  \\
     &=& {\rm Tr}_{n-1}(Y{\rm tr}_n(X))\quad \mbox{ (by induction hypothesis)}\\
     &=& {\rm Tr}_{n-1}({\rm tr}_n(YX))\quad \mbox{ (by (ii) Lemma \ref{traza1})}.
  \end{eqnarray*}
 Hence, $ {\rm Tr}_n(XY) =  {\rm Tr}_n(YX)$ for all $X\in \Y$ and $Y\in \Z$.
 Now, we prove the rule for   $Y\in \{g_{n-1}, t_n\}$. By using Lemma \ref{trazaimp}, we get
$$
{\rm Tr}_n(XY) = {\rm Tr}_{n-2}({\rm tr}_{n-1}({\rm tr}_n(XY)))
= {\rm Tr}_{n-2}({\rm tr}_{n-1}({\rm tr}_n(YX)))
$$
Summarizing, we have
$$
{\rm Tr}_n(XY) = {\rm Tr}_n(XY)
$$
for all $X\in \Y$ and $Y\in \Z \cup \{g_{n-1}, t_n\}$. Clearly, having in mind the linearity of ${\rm Tr}_n$, this last equality implies that rule (v) holds.
\end{proof}

The rules of the trace on the topological level are illustrated in the next figure.
\begin{figure}[h]
\begin{center}
  \includegraphics{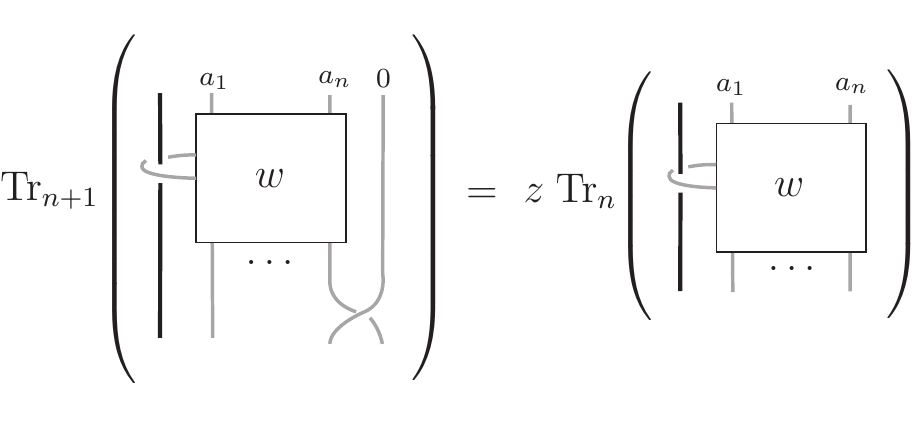}
		\end{center}
 \end{figure}
 \newpage
 \begin{figure}[h]
\begin{center}
  \includegraphics{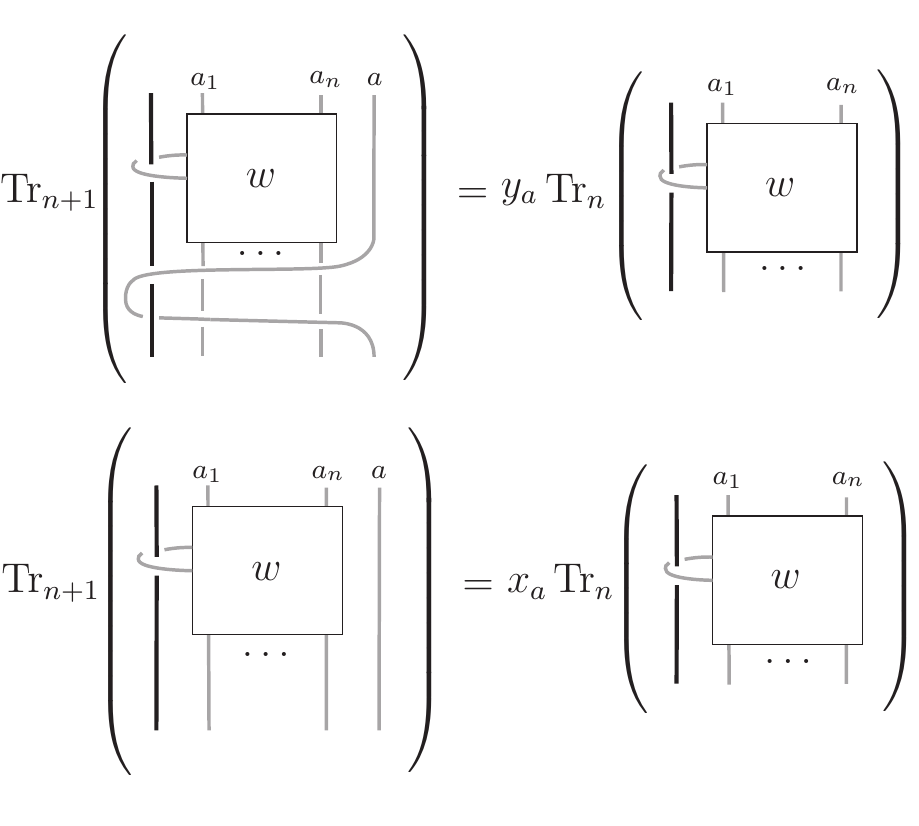}
	\caption{Trace rules (ii)-(iv).}\label{tracerule2}
	\end{center}
 \end{figure}
\section{The ${\rm E}$--condition and the ${\rm F}$--condition}

In this section we establish the necessary and sufficient  conditions by which the parameters trace $x_1, \dots, x_{d-1}, y_0, \dots, y_{d-1}\in \Bbb{L}$ satisfies the following equation.
 \begin{equation*}
 {\rm Tr}_{n+1}(we_n)={\rm Tr}_n(w){\rm Tr}_{n+1}(e_n)\qquad \mbox{for all $w\in \Y$}.
\end{equation*}
This equation plays a key role for defining knot and link invariants in the next section. In this section we will prove  that if the parameters satisfy the so--called  ${\rm E}$--conditon, and a set of new conditions,  called ${\rm F}$--condition, then the above equation holds; see Theorem \ref{factortr}. Finally, we will compute  such trace parameters, by using the method due to P. G\'erardin to solve the so--called ${\rm E}$--system, see \cite[Appendix]{jula2}.

\subsection{}

In \cite{jula2} certain elements $E^{(k)}$ were introduced, associated to the trace parameters  of the trace on the Yokonuma--Hecke algebra. With these $E^{(k)}$ the authors defined a non-linear system of equations called the {\it E--system}. We say that  the solutions of this ${\rm E}$--system satisfy the ${\rm E}$--{\it condition}. Notably, whenever the trace parameters  of the Markov trace on the Yokonuma--Hecke algebra satisfy the ${\rm E}$--condition we have an invariant for framed and classical knots and links.

\smallbreak

 We consider here the same formal expressions of elements $E^{(k)}$ associated now to the trace parameters $x_1,\ldots , x_{d-1}$ of ${\rm Tr}_n$. More precisely, we define
\begin{equation}\label{E^k}
 E^{(k)}:= \frac{1}{d}\sum_m x_{k+m}x_{d-m} \qquad \text{for}\qquad 0\leq k\leq d-1.
\end{equation}
Note that $E^{(0)} = {\rm Tr}_n(e_n)$. Also we need to introduce the following elements
\begin{equation}\label{F^k}
  F^{(k)}:=\frac{1}{d}\sum_m x_{d-m}y_{k+m} \qquad \text{for}\qquad 0\leq k\leq d-1.
\end{equation}

In the summations above the $m$'s are regarded modulo $d$.

\smallbreak
The ${\rm E}$--system is defined as the non--linear system of equations in $x_1,\dots, x_{d-1}$ formed by the following  $d-1$   equations:
\begin{eqnarray*}
  E^{(m)} &=& x_mE^{(0)}
\end{eqnarray*}
where $1\leq m\leq d-1$.  Any solution $({\rm x}_1,\dots, {\rm x}_n)$ of the ${\rm E}$--system is referred to by saying that it satisfies the ${\rm E}$--condition.

\smallbreak
Assume that $({\rm x}_1,\dots, {\rm x}_n)$ satisfies the ${\rm E}$--condition. The ${\rm F}$--system is the homogeneous linear system of equations in $y_0,\dots, y_{d-1}$, formed by the following $d$ equations:
\begin{eqnarray*}
  {\rm F}^{(m)} &=& y_m{\rm E}^{(0)}
\end{eqnarray*}
where $0\leq m\leq d-1$, and ${\rm E}^{(0)}$ and ${\rm F}^{(m)}$ are the elements that result from replacing $x_i$ by ${\rm x}_i$ in (\ref{E^k}) and (\ref{F^k}) respectively, that is:
$$ {\rm F}^{(m)}:=\frac{1}{d}\sum_m {\rm x}_{d-m}y_{k+m} \quad \text{and} \quad  {\rm E}^{(0)}:= \frac{1}{d}\sum_m {\rm x}_{m}{\rm x}_{d-m}.$$
Also we have that ${\rm E}^{(0)}=\frac{1}{|S|}$, see \cite[Section 4.3]{jula5}. Thus the ${\rm F}$--system is formed by the following equations
\begin{equation}\label{Fsystem}
  \sum_m {\rm x}_{d-m}y_{k+m}-\frac{d}{|S|}y_m=0 \qquad 0\leq m \leq d-1.
\end{equation}

Notice that the matrix associated to this linear system is given by:
$$\left(
    \begin{array}{cccc}
      {\rm x}_0-\frac{d}{|S|} & {\rm x}_{d-1}  & \dots  & {\rm x}_1 \\
      {\rm x}_1 & {\rm x}_0-\frac{d}{|S|} & \ddots  & {\rm x}_2   \\
      \vdots & \ddots &  \ddots &  {\rm x}_{d-1}\\
       {\rm x}_{d-1}& \dots &  {\rm x}_1& {\rm x}_0-\frac{d}{|S|}  \\
    \end{array}
  \right).
$$

Any solution $({\rm y}_0,\dots, {\rm y}_n)$ of the ${\rm F}$--system is referred to saying that it satisfies the ${\rm F}$--condition.
\smallbreak
We have the following theorem.

\begin{theorem}\label{factortr}
 We assume that the trace parameters are specialize to complex numbers $({\rm x}_1,\dots, {\rm x}_n)$ and  $({\rm y}_0,\dots, {\rm y}_n)$ that satisfy the ${\rm E}$--condition and the ${\rm F}$--condition respectively. Then
 \begin{equation}\label{trazah}
         {\rm Tr}_{n+1}(we_n)={\rm Tr}_n(w){\rm Tr}_{n+1}(e_n)\qquad \text{ for all} \qquad w\in \Y.
       \end{equation}
\end{theorem}

We shall prove this theorem at the end of the subsection and using  the Lemmas \ref{h1}--\ref{h3} below. We will introduce first the elements
$$e_n^{(m)}:=\frac{1}{d}\sum_{s=0}^{d-1}t_n^{m+s}t_{n+1}^{d-s}$$.

\begin{lemma}\label{h1}
  Let $w=w^{\prime}t_n^k$, where $w^{\prime}\in \Z$. Then
$$
{\rm Tr}_{n+1}(we_n^{(m)})=\frac{E^{(k+m)}}{x_k}{\rm Tr}_n(w).
$$
Hence,  ${\rm Tr}_{n+1}(we_n)=\frac{E^{(k)}}{x_k}{\rm Tr}_n(w)$.
\end{lemma}
  \begin{proof}
 Splitting $e_n^{(m)}$, we have:
 $$
 {\rm Tr}_{n+1}(we_n^{(m)}) =\frac{1}{d}\sum_s{\rm Tr}_{n+1}(w^{\prime}t_{n}^{k+m+s}t_{n+1}^{d-s}).
 $$
 Now, ${\rm Tr}_{n+1}(w^{\prime}t_{n}^{k+m+s}t_{n+1}^{d-s}) = x_{d-s}{\rm Tr}_n(w^{\prime}t_{n}^{k+m+s})=
 x_{d-s}x_{k+m+s} {\rm Tr}_n(w^{\prime})$. Then
$$
 {\rm Tr}_{n+1}(we_n^{(m)}) =
 \frac{1}{d}\sum_s x_{d-s}x_{k+m+s} {\rm Tr}_{n-1}(w^{\prime}) = E^{(k+m)}  {\rm Tr}_{n-1}(w^{\prime})
 = \frac{E^{(k+m)}}{x_k}{\rm Tr}_n(w).
 $$
  \end{proof}

\begin{lemma}\label{h2}
  Let $w=w^{\prime}b_nt_n^k$, where $w^{\prime}\in \Z$. Then
$$
{\rm Tr}_{n+1}(we_n^{(m)})=\frac{F^{(k+m)}}{y_k}{\rm Tr}_n(w).
$$
In particular, we have ${\rm Tr}_{n+1}(we_n)=\frac{F^{(k)}}{y_k}{\rm Tr}_n(w)$.
\end{lemma}
\begin{proof} We have:
   \begin{eqnarray*}
         {\rm Tr}_{n+1}(we_n^{(m)}) &=& \frac{1}{d}\sum_s{\rm Tr}_{n+1}(w^{\prime}b_{n}t_{n}^{k+m+s}t_{n+1}^{d-s}) \\
       &=& \frac{1}{d}\sum_s x_{d-s}{\rm Tr}_n(w^{\prime}b_{n}t_{n}^{k+m+s}) \\
       &=& \frac{1}{d}\sum_s x_{d-s}y_{k+m+s} {\rm Tr}_{n-1}(w^{\prime})  \\
       &=& F^{(k+m)}  {\rm Tr}_{n-1}(w^{\prime})\\
       &=& \frac{F^{(k+m)}}{y_k}{\rm Tr}_n(w).
    \end{eqnarray*}
\end{proof}

\begin{lemma}\label{h3}
  Let $w=w^{\prime}\mathfrak{m}_{n,k,\alpha}^{\pm}$, with $w^{\prime}\in \Z$. Then  ${\rm Tr}_{n+1}(we_n)=z{\rm Tr}_n(xe_{n-1})$, where $x=\mathfrak{m}_{n-1,k,\alpha}^{\pm}w^{\prime}$.
\end{lemma}
\begin{proof}We have:
   \begin{eqnarray*}
         {\rm Tr}_{n+1}(we_n) &=& \frac{1}{d}\sum_s{\rm Tr}_{n+1}(w'\mathfrak{m}_{n,k,\alpha}^{\pm}t_n^st_{n+1}^{d-s}) \\
       &=&  \frac{1}{d}\sum_s x_{d-s}{\rm Tr}_n(w'g_{n-1}\mathfrak{m}_{n-1,k,\alpha}^{\pm}t_n^s) \\
       &=&  \frac{1}{d}\sum_s x_{d-s}{\rm Tr}_n(w't_{n-1}^s g_{n-1}\mathfrak{m}_{n-1,k,\alpha}^{\pm}) \\
       &=&  \frac{z}{d}\sum_s x_{d-s}{\rm Tr}_{n-1}(w't_{n-1}^s \mathfrak{m}_{n-1,k,\alpha}^{\pm})\\
       &=&  \frac{z}{d}\sum_s x_{d-s}{\rm Tr}_{n-1}(\mathfrak{m}_{n-1,k,\alpha}^{\pm}w't_{n-1}^s) \\
       &=& \frac{z}{d}\sum_s {\rm Tr}_n(\mathfrak{m}_{n-1,k,\alpha}^{\pm}w't_{n-1}^st_{n}^{d-s})=z{\rm Tr}_n(xe_{n-1}).
    \end{eqnarray*}
\end{proof}

\begin{proof}[Proof of Theorem \ref{factortr}]
  By the linearity of ${\rm Tr}_{n+1}$ we can assume that $w$ is an element in the inductive basis $\mathsf{C}_n$. We proceed by induction on $n$. For $n=1$ we have two possibilities: $w=t_1^{k}$ or $w=b_1t_1^{k}$.
 For  $w=t_1^{k}$, we have:
 $$
 {\rm Tr}_{n+1}(we_1)=\frac{1}{d}\sum_s{\rm x}_{d-s}{\rm x}_{k+s}=\frac{{\rm E}^{(k)}}{{\rm x}_k}{\rm Tr}_n(w) = {\rm E}^{(0)}{\rm Tr}_n(w)={\rm Tr}_{n+1}(e_1){\rm Tr}_n(w).
 $$
For  $w=b_1t_1^{k}$, we have:
 $$
 {\rm Tr}_{n+1}(we_1)=\frac{1}{d}\sum_s{\rm x}_{d-s}{\rm y}_{k+s}=\frac{{\rm F}^{(k)}}{{\rm y}_k}{\rm Tr}_n(w)=  {\rm E}^{(0)}{\rm Tr}_n(w)={\rm Tr}_{n+1}(e_1){\rm Tr}_n(w).
 $$
Thus,  for $n=1$ the theorem is proved. Suppose now that the theorem  is true for every positive integer less than $n+1$. Set $w$ be an element in $\mathsf{C}_n$. We shall prove the theorem by distinguishing the  three types of form for $w$.

\smallbreak
\noindent $\bullet$ Suppose $w=w^{\prime}t_n^k$, where $w^{\prime}\in \Z$. By using the Lemma \ref{h1} and the fact that ${\rm x}_k$'s satisfies the ${\rm E}$--condition, we have:
$$
{\rm Tr}_{n+1}(we_n)=\frac{{\rm E}^{(k)}}{{\rm x}_k}{\rm Tr}_n(w)={\rm E}^{(0)}{\rm Tr}_n(w)={\rm Tr}_n(w){\rm Tr}_{n+1}(e_n).
$$

\noindent
$\bullet$ Suppose $w=w^{\prime}b_nt_n^k$, where $ w^{\prime}\in \Z$. Then, by using now Lemma \ref{h2} and the fact that the ${\rm y}_k$'s satisfied the ${\rm F}$--condition,  we have:
$$
{\rm Tr}_{n+1}(we_n)=\frac{{\rm F}^{(k)}}{{\rm y}_k}{\rm Tr}_n(w)={\rm E}^{(0)}{\rm Tr}_n(w)={\rm Tr}_n(w){\rm Tr}_{n+1}(e_n).
$$

\noindent
$\bullet$ Finally, suppose  $w=w^{\prime}\mathfrak{m}_{n,k,\alpha}^{\pm}$, where $w'\in \Z$. From  Lemma \ref{h3}, we have
$$
{\rm Tr}_{n+1}(we_n)=z{\rm Tr}_{n}(xe_{n-1})\quad \mbox{where $x=\mathfrak{m}_{n-1,k,\alpha}^{\pm}w'$}\in \Z.
$$
Now, by using the induction hypothesis,   we get ${\rm Tr}_n(xe_{n-1}) = {\rm Tr}_{n-1}(x){\rm Tr}_n(e_{n-1})$. But, now ${\rm Tr}_{n+1}(e_n)={\rm Tr}_n(e_{n-1})$ and
$$
z{\rm Tr}_{n-1}(x) = {\rm Tr}_n(g_{n-1}\mathfrak{m}_{n-1,k,\alpha}^{\pm}w^{\prime}) = {\rm Tr}_n(\mathfrak{m}_{n,k,\alpha}^{\pm}w^{\prime}).
$$
Therefore
$$
{\rm Tr}_{n+1}(we_n)={\rm Tr}_n(w){\rm Tr}_{n+1}(e_n).
$$
\end{proof}

\subsection{\it Solving  the  ${\rm F}$--system}

The ${\rm E}$--system was solved by P. G\'erardin, by using some tools from the complex harmonic analysis on finite groups, see \cite[Appendix]{jula5}. However, his method works on any field having characteristic $0$. We shall  introduce now some notations and definitions, necessary to explain the method used  by G\'erardin, which will be used  to solve the ${\rm F}$--system as well. For more details on the tools of harmonic analysis used here, see \cite{jula2, gojukolaLMJ}.

\smallbreak

We shall regard the group algebra $\Lambda :={\Bbb L}[{\Bbb Z}/d{\Bbb Z}]$, as the algebra formed by all complex functions on ${\Bbb Z}/d{\Bbb Z}$, where the product is the  convolution product, that is:
$$
(f \ast g)(x)=\sum_{y\in {\Bbb Z}/d{\Bbb Z} }f(y)g(x-y)\qquad \text{where} \quad f, g\in \Lambda.
$$
As usual, we denote by $\delta_a \in \Lambda$ the function with support $\{a\}$. Recall that $\delta_0$ is the unity with respect to the convolution product and that $\{\delta_a\ ; \ a\in {\Bbb Z}/d{\Bbb Z}\}$ is a linear basis for $\Lambda$. The algebra $\Lambda$ is commutative and is the direct sum of the simple ideals ${\Bbb K}\,{\bf e}_{a}$, where $a\in {\Bbb Z}/d{\Bbb Z} $ and the ${\bf e}_{a}$'s are the characters of ${\Bbb Z}/d{\Bbb Z} $, that is:
$$
{\bf e}_a:b\mapsto \cos \left(\frac{2\pi ab}{d}\right) + i\sin \left(\frac{2\pi ab}{d}\right).
$$

In $\Lambda$ we have another product, the punctual product, that is:
$$
fg: x\mapsto f(x)g(x)\qquad \text{where} \quad f, g\in \Lambda.
$$
The algebra $\Lambda$ with the punctual product has unity ${\bf e}_0$ and is the direct sum of its simple ideals
${\Bbb K}\,\delta_a$, where $a\in {\Bbb Z}/d{\Bbb Z}$.
\smallbreak
The Fourier transform $\mathcal{F}$ on $\Lambda$ is the automorphism defined by $f\mapsto \widehat{f}$, where
$$
\widehat{f}(x):= (f\ast {\bf e}_x)(0)=\sum_{y\in {\Bbb Z}/d{\Bbb Z}}f(y){\bf e}_x(-y).
$$
Recall that $(\mathcal{F}^{-1}f)(x) = d^{-1} \widehat{f}(-u)$,
where $\widehat{f} (v)=\sum_{u\in G}f(u){\bf e}_v(-u)$.
\smallbreak
The following proposition collects the properties of the Fourier transform used here. These properties are well-known and can be found, for example, in \cite{te}.

\begin{proposition}\label{fourier}
For every $a\in {\Bbb Z}/d{\Bbb Z}$ and $f, g\in \Lambda$. We have:
\begin{enumerate}
  \item[(i)] $\widehat{\delta_a}={\bf e}_{-a}$
  \item[(ii)] $\widehat{{\bf e}_a}=d\delta_a$
  \item[(iii)] $\widehat{\widehat{f}\,\,}(u)=df(-u)$
  \item[(iv)] $\widehat{f\ast g}=\widehat{f}\, \widehat{g}$
  \item [(v)]$\widehat{fg}=d^{-1}\widehat{f}\ast \widehat{g}$.
\end{enumerate}
\end{proposition}

 To solve the ${\rm E}$--system, G\'erardin considered the elements $x\in \Lambda$, defined by $x(k)= x_k$. Then, he interpreted the ${\rm E}$--system as the functional equation $x\ast x = (x\ast x)(0) x$ with the initial condition $x(0)=1$.
Now, by applying the Fourier transform on this functional equation we obtain $\widehat{x}^2 = (x\ast x)(0) \widehat{x}$. This last equations implies that $\widehat{x}$ is constant on its support $S$, where it takes the values $(x\ast x)(0)$. Thus, we have
$$
\widehat{x} = (x\ast x)(0)\sum_{s\in S}\delta_s.
$$
By applying $\mathcal{F}^{-1}$ and the properties listed in the proposition above, G\'erardin showed that the solutions of the ${\rm E}$--system are parameterized by the non--empty subsets of ${\Bbb Z}/d{\Bbb Z}$. More precisely, for such a
subset $S$,  the solution $x_S$ is given as follows.
$$
x_S=\frac{1}{|S|}\sum_{s\in S}{\bf e}_s.
$$

Now, in order to solve the ${\rm F}$--system with respect to $x_S$, we define $y\in\Lambda$ by $y(k)=y_k$. Then we have
$F^{(k)} = d^{-1} (x\ast y)(k)$. So, to solve the ${\rm F}$--system is equivalent to solving the following functional equation:
$$
x\ast y = (x\ast x)(0)y.
$$
which, applying the Fourier transform and Proposition \ref{fourier} (iv), is equivalent to:
$$
\widehat{x}\widehat{y}=(x\ast x)(0)\widehat{y}.
$$
This equation implies that the support of $\widehat{y}$ is contained in the support of $\widehat{x}$.  Now, set $S$ the support of $\widehat{x}$. Then we can write $\widehat{y}=\sum_{s\in S}\lambda_s\delta_s$. In this last equation, by
applying $\mathcal{F}^{-1}$ and  Proposition \ref{fourier} (i) and (iv), we get:
$$
y= \frac{1}{d}\sum_{s\in S}\lambda_s{\bf e}_{s}.
$$
Thus, we have proved the following proposition.

\begin{proposition}
The solution of the ${\rm F}$-system with  respect to the solution $x_S$ of the ${\rm E}$--system is in the form:
$$
y_S= \sum_{s\in S}\alpha_s{\bf e}_{s}
$$
where the $\alpha_s$'s are complex numbers.
\end{proposition}

\section{Knot and link invariants from $\Y$}

In this section we define invariants for knots and links in the solid torus, by using the Jones recipe applied to the pairs $(\Y, {\rm Tr}_n)$ where $n\geq 1$. To do that, we fix from now on  that the trace parameters $x_k$ satisfy  the ${\rm E}$--condition and the trace parameters  $y_k$ satisfy the ${\rm F}$--system, with respect to
the $x_k$'s. The invariants constructed here will take values in $\Bbb{L}$.

More precisely, the closure of a framed braid $\alpha$ of type ${\mathtt B}$ (recall Section~\ref{framizations}) is defined by joining with simple (unknotted and unlinked) arcs its corresponding endpoints and is denoted by $\widehat{\alpha}$. The result of closure, $\widehat{\alpha}$, is a framed link in the solid torus, denoted $ST$. This can be understood by viewing the closure of the fixed strand as the complementary solid torus. For an example of a framed link in the solid torus see Figure~\ref{flink}. By the analogue of the Markov theorem for  $ST$ (cf. for example \cite{la1,la2}), isotopy classes of oriented links in $ST$ are in bijection with equivalence classes of braids of type ${\mathtt{B}}$ and this bijection  carries through to the class of framed links of type ${\mathtt{B}}$.
\begin{figure}[h!]
\begin{center}
  \includegraphics{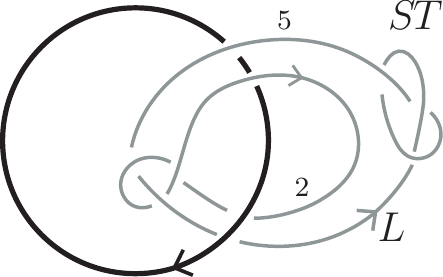}
	\caption{A framed link in the solid torus.}\label{flink}
	\end{center}
 \end{figure}

We set
\begin{equation}\label{CapitalLambda}
\lambda_S := \frac{z - ({\rm u}- {\rm u}^{-1}){\rm E}_S}{z} \quad \text{and} \quad \Lambda_S:=\frac{1}{z \sqrt{\lambda_S}},
\end{equation}
where ${\rm E}_S= 1/|S|$. We are now in the position to define link invariants in the solid torus.

\begin{definition} \rm
 For $\alpha$ in $\mathcal{F}_{n}^{\mathtt{B}}$, the Markov trace ${\rm Tr}$ with the trace parameters specialized to solutions of the ${\rm E}$--system and the ${\rm F}$--system, and $\pi$ the natural epimorphism of $\mathcal{F}_{n}^{\mathtt{B}}$ onto $\Y$ we define
$$
{\mathcal X}_S^{\mathtt{B}}(\widehat{\alpha})   :=
 \Lambda_S^{n-1} (\sqrt{\lambda_S})^e\, {\rm Tr}(\pi(\alpha)),
$$
\noindent where $e$ is the exponent sum of the $\sigma_i$'s that appear in  $\alpha$.   Then
${\mathcal X}_S^{\mathtt{B}}$  is a Laurent polynomial in ${\rm u}, {\rm v}$ and $z$ and it  depends    only  on     the   isotopy  class   of the  framed link
$\widehat{\alpha}$, which represents an oriented  framed link in $ST$.
\end{definition}

\begin{remark} \rm
The invariants ${\mathcal X}_S^{\mathtt{B}}$, when restricted to framed links with all framings equal to 0, give rise to invariants of oriented classical links in  $ST$. By the results in \cite{chjukala} and since classical knot theory embeds in the knot theory of the solid torus, these invariants are distinguished from the Lambropoulou invariants \cite{gela, la1}. More precisely,  they are not topologically equivalent to these invariants on {\it links}.
\end{remark}

\begin{remark} \rm
As we have said previously the cyclotomic Yokonuma--Hecke algebra ${\rm Y}(d,m,n)$ provides a framization of the Hecke algebra of type $\mathtt{B}$ when $m=2$. In \cite{chpoIMRN} where this algebra was introduced, it was also proved that ${\rm Y}(d,m,n)$ supports a Markov trace, which will be denoted here by $\mathtt{Tr}$, for details see \cite[Section 5]{chpoIMRN}. Then using Jones's recipe a new invariant for framed links in the solid torus is constructed, which is given by
\begin{equation}\label{inv2}
 \Gamma_m(\widehat{\alpha}):=\Lambda_S^{n-1} (\sqrt{\lambda_S})^e\, {\mathtt Tr}(\overline{\pi}(\alpha)),
\end{equation}
where $\overline{\pi}:\mathcal{F}_{n}^{\mathtt{B}}\rightarrow {\rm Y}(d,2,n)$ is the natural algebra epimorphism given by
$$\rho_1\mapsto b_1,\quad \sigma_i\mapsto g_i,\quad i=1,\dots, n-1,\quad \text{and}\quad t_j\mapsto t_j,$$ see \cite[Section 6.3]{chpoIMRN}.
 As we see in the previous section, in order that this polynomial becomes an invariant, the trace parameters (of $\mathtt{Tr}$) have to satisfy a non-linear system of equations, which for $m=2$ is equivalent to the systems given here (E-- and F--system).\\

 Now, we would like to make some comparison between $\Gamma_2$ and ${\mathcal X}_S$. At first sight the invariants look similar, but the structural differences between the $\Y$ and ${\rm Y}(d,2,n)$ make them differ (see Remark \ref{diff}). For example, for the loop generator twice, we have the following
\begin{center}
\begin{tabular}{r l|r l}
&\text{In $\Y$}&& \text{In ${\rm Y}(d,2,n)$}\\
\hline
   ${\rm Tr}(\pi(b_1^2))=$&${\rm Tr}(1+(\V-\V^{-1})b_1f_1)$ &${\mathtt Tr}(\overline{\pi}(b_1^2))=$&${\mathtt Tr}(1+(\V-\V^{-1})b_1)$ \\
    =&$1+\frac{(\V-\V^{-1})}{d}\sum_{s}{\rm Tr}(b_1t_1^s)$ & =&$1+(\V-\V^{-1})y_0$ \\
  =&$1+\frac{(\V-\V^{-1})}{d}\sum_{s}y_s$ & &
\end{tabular}
\end{center}
Therefore
$${\mathcal X}_S^{\mathtt{B}}(\widehat{b_1^2})=1+\frac{(\V-\V^{-1})}{d}\sum_{s}{\rm y}_s\quad \text{and} \quad \Gamma_2(\widehat{b_1^2})=1+(\V-\V^{-1}){\rm y}_0 $$
\end{remark}
Then clearly for the framed link $\widehat{b_1^2}$, the two invariants have different values, nevertheless to do a proper comparison of these invariants it is necessary a deeper study.

\end{document}